\documentclass[12pt]{article}
\usepackage{mathrsfs}
\usepackage{graphicx}
\usepackage{enumerate}
\usepackage{amsthm,amsmath,amsfonts,amssymb}
\usepackage{bm}
\usepackage{color}
\usepackage{subfigure}
\usepackage{multirow}

\allowdisplaybreaks[3]
\usepackage[numbers,sort&compress]{natbib}
\usepackage{dsfont}
\usepackage{ulem}

\theoremstyle{plain}

\newtheorem{theorem}{Theorem}[section]
\newtheorem{lemma}{Lemma}[section]
\newtheorem{corollary}{Corollary}[section]
\newtheorem{proposition}{Proposition}[section]

\theoremstyle{remark}
\newtheorem{definition}{Definition}

\newtheorem{remark}{Remark}

\newcommand{\dint}{\displaystyle\int}

\newcommand{\suml}{\sum\limits}

\newcommand{\maxl}{\max\limits}
\newcommand{\minl}{\min\limits}
\newcommand{\supl}{\sup\limits}

\newcommand{\df}{\mathrm{d}}

\newcommand{\dt}{\mathrm{d}t}
\newcommand{\bs}{\boldsymbol}

\renewcommand{\(}{\left(}
\renewcommand{\)}{\right)}

\newcommand{\Ep}{{\mathbb{E}}}

\renewcommand{\Pr}{{\mathbb{P}}}

\newcommand{\RR}{\mathbb{R}}
\newcommand{\les}{\lesssim}

\newcommand*\supercite[1]{\textsuperscript{\cite{#1}}}
\newcommand{\rulex}{\hfill\rule{1mm}{3mm}}

\DeclareMathOperator{\Var}{Var}

\newcommand{\s}{\mathfrak{s}}

\usepackage{hyperref}

\begin{document}
	
	\begin{center}
		{\bf\Large Gaussian Multiplier Bootstrap Procedure for the $k$th Largest Coordinate of High-Dimensional Statistics}
		\vspace{3mm}
		
		Yixi Ding$^{1,2}$, Qizhai Li$^{1,2}$, Yuke Shi$^{3}$, Liuquan Sun$^{1,2}$, Luobin Zhang$^{1,2}$
		
		$^1$ State Key Laboratory of Mathematical Sciences, Academy of Mathematics and Systems Science, Chinese Academy of Sciences, Beijing 100190, China
		
		$^2$ School of Mathematical Science, University of Chinese Academy of Sciences, Beijing 101408, China
		
		$^3$ Department of Information and Computing Science, University of Science and Technology Beijing, Beijing, China
		
	\end{center}

	\author{Yixi Ding, Qizhai Li, Yuke Shi, Liuquan Sun, Luobin Zhang}

		\noindent{\textbf{Abstract}} ~ We consider the problem of Gaussian multiplier bootstrap procedures for the $k$th largest statistics and functions of the top $k$ order statistics, which are  commonly encountered in high-dimensional statistical inference.
		Such a problem  has been studied previously for $k=1$ (i.e., maxima). However, in many applications, a general $k$ ($k\geq 1$) is of great interest.
		We provide the upper bounds for the errors between Gaussian approximations and Gaussian multiplier approximations.
		The dimension $p$  is allowed to be larger than the sample size $n$. The effectiveness of the proposed methods is demonstrated  via the computer numerical results and a real-world data analysis.
		
	
	\section{Introduction}
	
	High dimensional hypothesis testing problems are commonly encountered in statistical theories and applications, for example,  testing whether the coefficients in linear regression\supercite{Belloni2014,Gao2022,Guo2022}   and generalized linear regression\supercite{Zhang2017, Ning2017, Li2024}
	are zero.
	The distributions or asymptotic distributions of test statistics are mainly derived based on the sparsity assumption of the covariance matrix.
	However, if this postulation is violated, obtaining their distributions or asymptotic distributions is extremely challenging.
	Gaussian multiplier bootstrap methods\supercite{GAR,Chen2020,Lopes2020} provide accurate approximations for the maximum of sums of high-dimensional independent random vectors.
	Whether this validity still holds for the $k$th largest order statistics and the functions of top $k$ order statistics is still unclear.
	We point out that taking the combination of top $k$ order statistics is very common in practice since only one signal in high dimensional problems is highly improbable.

	Let $\bs{X}_{1},\ldots,\bs{X}_n$ be $n$ independent random vectors in $\mathbb{R}^p$  with $\bs{X}_{i} = (X_{i1},\ldots,X_{ip})^{\rm T}$ for $i=1,\ldots,n$, where the superscript ${\rm T}$
	indicates the transpose of a matrix or a vector.
	Assume each unit in $\boldsymbol X_i $ is centered and has a finite variance, that is, $ \mathbb{E}[X_{ij}] = 0 $, and $ \mathbb{E}[X_{ij}^2] < \infty $ for $ i = 1,\ldots, n $ and $ j = 1, \ldots, p $, where $\mathbb E$ denotes the operator of expectation.
	Denote $\boldsymbol T = (T_1, \ldots, T_p)^{\rm T} = \frac{1}{\sqrt{n}} \sum_{i=1}^n \boldsymbol X_i$.
	We rank $T_1,\ldots,T_p$ and $|T_1|,\ldots,|T_p|$ separately in a decreasing order, and denote the order statistics by $T_{(1)},\ldots,T_{(p)}$ and $T_{[1]},\ldots,T_{[p]}$, respectively.
	Thus $T_{(1)}\ge\cdots\ge T_{(p)}$ and $T_{[1]}\geq\cdots\geq T_{[p]}$.
	Let $\psi$ be a $k$-dimensional measurable function defined in $\mathbb R^k$ for a  fixed positive integer $k$. We are interested in the asymptotic distribution of $T_{(k)}$, $T_{[k]}$ and $T_\psi$ with $T_\psi=\psi(T_{[1]},\ldots,T_{[k]})$, when both $n$ and $p$ go to infinity.
	
	When the independent and identically distributed ({\it iid}) observations are sampled from an exact distribution, Fisher and Tippett\supercite{Fisher1928} obtained the extreme value distribution of the maximal order statistic.
	Mu\supercite{Mu1966} gave the asymptotic normal distribution of the intermediate order statistic under some conditions.
	Watts, Rootzen, and Leadbetter\supercite{Watts1982} established the asymptotic normal distribution of an intermediate order statistic from a stationary dependent sequence.
	As for the asymptotic distribution,  Chernozhukov, Chetverikov, and Kato\supercite{GAR} gave the Gaussian multiplier bootstrap bounds for the maximum of sums of random vectors.
	Their subsequent work\supercite{CCK17,ICL}  extended these results.
	Lopes, Lin, and Müller\supercite{Lopes2020} studied the convergence rate of Gaussian multiplier bootstrap approximation by considering some extra conditions.
	Chen and Kato\supercite{Chen2020} applied Gaussian multiplier bootstrap to approximate the supremum of a nondegenerate U-process.
	
	We study Gaussian and Gaussian multiplier bootstrap approximations for $T_{(k)}$, $T_{[k]}$ and $T_\psi$ with fixed $k$.
	The upper bounds of the Kolmogorov-Smirnov distances between the asymptotic distributions of $T_{(k)}$, $T_{[k]}$, and $T_\psi$ and   their  Gaussian multiplier bootstrap analogs are given.
	The resulting bounds depend on $p$ only through logarithmic factors, allowing $p$ to be much larger than $n$ under the stated moment conditions.

	We list some notations. For  a  function $W(\cdot): \mathbb{R}^p\to \mathbb{R}$,  let $\partial_{j} W (\boldsymbol{x}) = \partial W(\boldsymbol{x})/ \partial x_{j}$ be the partial derivative of $W(\cdot)$ with respect to $x_j$, $j=1,\ldots,p$, where $\boldsymbol{x}= (x_{1},\ldots,x_{p})^{\top}\in{\mathbb R}^p$.
	For a function $w(t)$ defined on $\mathbb R$, let $w^{(j)}(t)=\partial^j w(t)=\partial^j w(t)/\partial t^j$ for $t\in \mathbb{R}$, where $j$ is a positive integer.
	We use ${\mathbb C}_{m}(\mathbb{R})$ to denote the class of functions $\tilde \psi$ satisfying $\sup_{t\in \mathbb{R}} | \partial^{j} \tilde \psi(t) | < \infty$ for $j=0,\ldots,m$.
	For two positive constants $a$ and $b$, we write $a \lesssim b$ if $a$ is smaller than or equal to $b$ up to a universal positive
	constant, and $a\gtrsim b$ vice versa. Denote $a \vee b = \max \{ a,b \}$, $a\wedge b=\min \{ a,b \} $, and $a_+=\max\{a,0\}=a\vee 0$.
	For a random variable $T$ and constant $\gamma\in(0,1)$, the $\gamma$th quantile of $T$ is defined as  $\inf\{t\in\mathbb R\colon  \mathbb P(T\leq t) \ge \gamma \}$, where $\mathbb P$ denotes the operator of taking probability.
	Let $\|T\|_{\rm e} = \inf\left\{\lambda > 0 : \mathbb{E} \left[\exp(|T|/\lambda)\right] \leq 2\right\},$	and  $\mathds{1}\{\cdot\}$ be an indicator function. For $\boldsymbol v=(v_1, \ldots, v_p)^{\rm T} \in \mathbb{R}^p$,
	let $\|\boldsymbol v\|_0 = \sum_{j=1}^p \mathds{1}\{v_j\neq 0\},$ $\|\boldsymbol v\|_2 = \big(\sum_{j=1}^p v_j^2\big)^{1/2}$, and $\|\boldsymbol v\|_\infty = \max_{j }  |v_j|$.
	We rank $v_1,\ldots,v_p$ in a decreasing order and denote them by $v_{(1)},\ldots, v_{(p)}$.
	For two vectors $\boldsymbol v_1,\boldsymbol v_2\in\mathbb R^p$ and $c\in\mathbb R$, we write $\boldsymbol v_1 \leq \boldsymbol v_2$ if each unit in $\boldsymbol v_1$  is not larger than the corresponding one in $ \boldsymbol v_2$, and  $\boldsymbol v_1+ c$  the vector obtained by adding $c$ to each unit of  $\boldsymbol v_1 $.


	\section{Smooth function}\label{S2-SAFSK}

	Assume $k\ll p/2$ and, for $\bs{x}=(x_1,\ldots,x_p)^{\rm T}\in\mathbb R^p$, define $\mathfrak{s}_k(\bs{x})=x_{(k)}$.
	The map $\mathfrak{s}_k$ and the indicator functions entering distribution functions are nonsmooth.
	The comparison arguments therefore use smooth approximations to both objects.
	Let $\beta>0$ be a smoothing parameter.
	For $k=1$, the standard log-sum-exp approximation\supercite{GAR} is $\mathfrak{f}_{p}(\bs{x})=\beta^{-1}\ln\{\sum_{j=1}^{p}\exp(\beta x_{j})\}$.
	For $k>1$, we use the smooth approximation of Liu and Xie\supercite{AEG}.
	Let $\mathscr{A}_k=\{\{s_1,\ldots,s_k\}:1\le s_1<\cdots<s_k\le p\}$ be the collection of all $k$-subsets of $\{1,\ldots,p\}$.
	Denote $\mathscr{A}_k(j)=\{A:j\in A,A\in\mathscr{A}_k\}$.
	For $A\in\mathscr{A}_k$, define
	\begin{align*}
		\mathfrak{q}_k(\bs{x})=\frac{1}{\beta}\ln \bigg( \sum_{A\in \mathscr{A}_k}\exp\big(\beta g_{k,A} (\bs{x})\big)\bigg),
	\end{align*}
	where $g_{k,A}(\bs{x})=-\beta^{-1}\ln\{ \sum_{s\in A}\exp(-\beta x_s)\} +\beta^{-1}\ln k$ and $\tilde{g}_k(\bs{x})=\sum_{A\in\mathscr{A}_k}\exp\{\beta g_{k,A}(\bs{x})\}$.
	Then
	\begin{equation} \label{eq: f properties}
		0\le \mathfrak{q}_k(\bs{x})-\mathfrak{s}_k(\bs{x})\le \frac{\ln k+\ln \binom{p}{k}}{\beta}\le \beta^{-1}k\ln p.
	\end{equation}
	Thus $\mathfrak{q}_k$ is a smooth upper approximation to $\mathfrak{s}_k$ with approximation error controlled by $\beta^{-1}k\ln p$.

	Since $	\mathbb P(T_{(k)}<t)=\mathbb E[\mathds{1}\{T_{(k)}<t\}]$, we  introduce a smooth function $g$ to approximate $\mathds{1}\{\cdot\}$.
	Let $g_0\colon\mathbb R\to\mathbb R$ be fifth order continuously differentiable and a decreasing function satisfying
	(i) $g_0(t)\geq 0$ for  $t\in\mathbb R$,
	(ii) $g_0(t) = 0$ for  $t\geq 1$, and
	(iii) $g_0(t) = 1$ for  $t\leq 0$.
	For $g_0$, there exists a constant $C_g>0$ such that
	$$
	\sup_{t\in\mathbb R}\Big(|g_0^{(1)}(t)|\vee|g_0^{(2)}(t)|\vee|g_0^{(3)}(t)|\vee|g_0^{(4)}(t)|\vee|g_0^{(5)}(t)|\Big) \leq C_g.
	$$
	Let $\phi \geq 1$,
	$	\beta = \phi\ln p$ and   $g(t) = g_0(k^{-1}\phi t)$ for  $t\in\mathbb R$.
	It follows that
	\begin{equation}
		g(t)=\begin{cases}
			1, & \text{if }t\leq0,\\
			0, & \text{if }t\geq k\phi^{-1}.
		\end{cases}\label{eq: g property}
	\end{equation}
	
	For fixed $\bs{y}\in\mathbb R^p$, we define the function  
	\begin{align}\label{def: m}
		\mathfrak{m}:=\mathfrak{m}(\bs{w},\bs{y}) = g\big(\mathfrak{q}_k(\bs{w} - \bs{y})\big), \ \text{for  }\bs{w}\in\mathbb R^p.
	\end{align}
	Then $\mathfrak{m}$ has up to the fifth order derivatives. For brevity of notation, we will use indices to denote these derivatives. For example, for $j_1,j_2,j_3,j_4,j_5=1,\ldots,p$, we write
	$$
	\mathfrak{m}_{j_1j_2j_3j_4j_5}(\bs{w},\bs y)=\frac{\partial^5	\mathfrak{m} (\bs{w},\bs y)}{\partial w_{j_1} \partial w_{j_2} \partial w_{j_3} \partial w_{j_4} \partial w_{j_5}}, \ \bs{w}\in\mathbb R^p.
	$$
	Let  the cardinality of $\mathscr{A}_k$ be $N=|\mathscr{A}_k|=\binom{p}{k}$, $\mathscr{A}_{k}=\{A_1,\ldots, A_N\},$
	and
	$\mathbf{H}:=\mathbf{H}(\bs{w},\bs y)=(g_{k,A_1}(\bs{w}-\bs{y}),\ldots,g_{k,A_N}(\bs{w}-\bs{y}))^{\rm T}.$
	The following properties corresponding to $\mathfrak{m}=g\circ \mathfrak{f}_{N}\circ \mathbf{H}$ are obtained.

	\begin{lemma}\label{Derivatives of m}
		For  $j_1,j_2,j_3,j_4,j_5=1,\ldots,p$, there exist functions $U_{j_1j_2} $, $U_{j_1j_2j_3} $, $U_{j_1j_2j_3j_4} $, and $U_{j_1j_2j_3j_4j_5} $ such that
		\vspace{2ex}
		
		(i) For  $\bs{w}\in\mathbb R^p$,
		\begin{equation*}\label{eq: property 1 - 2 and 3}
			|\mathfrak{m}_{j_1j_2}(\bs{w},\bs{y})| \leq U_{j_1j_2}(\bs{w},\bs{y}), \quad |\mathfrak{m}_{j_1j_2j_3}(\bs{w},\bs{y})| \leq U_{j_1j_2j_3}(\bs{w},\bs{y}),
		\end{equation*}
		\begin{equation*}\label{eq: property 1 - 3 and 4}
			|\mathfrak{m}_{j_1j_2j_3j_4}(\bs{w},\bs{y})| \leq U_{j_1j_2j_3j_4}(\bs{w},\bs{y}), ~\text{and} ~ |\mathfrak{m}_{j_1j_2j_3j_4j_5}(\bs{w},\bs{y})| \leq U_{j_1j_2j_3j_4j_5}(\bs{w},\bs{y});
		\end{equation*}
		
		(ii) For  $\bs{w}\in\mathbb R^p$,
		\begin{equation*}\label{eq: property 3 - 2 and 3}
			\sum_{j_1,j_2=1}^p U_{j_1j_2}(\bs{w},\bs{y}) \lesssim k^{-1}\phi^2\ln p, \quad \sum_{j_1,j_2,j_3=1}^p U_{j_1j_2j_3}(\bs{w},\bs{y}) \lesssim k^{-1}\phi^3\ln^2p,
		\end{equation*}
		\begin{equation*}\label{eq: property 3 - 3 and 4}
			\sum_{j_1,j_2,j_3,j_4=1}^p U_{j_1j_2j_3j_4}(\bs{w},\bs{y}) \lesssim k^{-1}\phi^4\ln^3p,   ~\text{and} ~ \sum_{j_1,j_2,j_3,j_4,j_5=1}^p U_{j_1j_2j_3j_4j_5}(\bs{w},\bs{y}) \lesssim k^{-1}\phi^5\ln^4p;
		\end{equation*}	
		
		(iii) For  $\bs{z}\in\mathbb R^p$ and $\bs{w}\in\mathbb R^p$ satisfying $\beta\|\bs{w}\|_{\infty}\leq 1$,
		\begin{equation*}\label{eq: property 2 - 4 and 5}
			U_{j_1j_2j_3j_4}(\bs{z}+\bs{w},\bs{y})\lesssim U_{j_1j_2j_3j_4}(\bs{z},\bs{y}), ~\text{and} ~ U_{j_1j_2j_3j_4j_5}(\bs{z}+\bs{w},\bs{y})\lesssim U_{j_1j_2j_3j_4j_5}(\bs{z},\bs{y}).
		\end{equation*}
		
	\end{lemma}
	The derivative bounds in Lemma \ref{Derivatives of m} are the smoothing inputs used in the comparison arguments below.

	\section{Theoretical preparations }
	The theoretical preparations consist of three ingredients for the Gaussian and multiplier-bootstrap approximations of $T_{(k)}$, $T_{[k]}$ and $T_\psi$.
	Subsection~\ref{sec: main arguments} gives a distributional comparison for smooth approximations of the $k$th order statistic.
	Subsection~\ref{sec:anti-concentration} states the anti-concentration bound for Gaussian order statistics.
	Subsection~\ref{sec: stein kernels} gives Stein-kernel comparisons used to control Gaussian and bootstrap laws.

	\subsection{ Distributional approximation}\label{sec: main arguments}
	
	Let $\bs{V}_1,\ldots,\bs{V}_n, \bs{Z}_1,\ldots,\bs{Z}_n$ be a sequence of $p$-dimensional independent random vectors with $\bs{V}_i$ and $\bs{Z}_i$ having coordinates $V_{i j}$ and $Z_{i j}$ in $\mathbb R$, respectively, such that $\mathbb E V_{i j} = \mathbb E Z_{i j} = 0$ for  $i = 1,\ldots,n$ and $j = 1,\ldots,p$. Let $B_n > 1$ be a sequence of constants that may diverge to infinity as $n\to\infty$.
	Assume the following conditions.
	
	\medskip
	\noindent
	{(L1)} {\it
		$
		\frac{1}{n}\sum_{i=1}^n\mathbb E \left[V_{i j}^4 + Z_{i j}^4\right] \lesssim  B_n^2, ~~ j = 1,\ldots,p,
		$
		and
		$$
		\mathbb E \Big[\|\bs{V}_i\|_{\infty}^8 + \|\bs{Z}_i\|_{\infty}^8\Big] \lesssim    B_n^8\ln^8(p n), ~~ i = 1,\ldots,n.
		$$
		
	}
	\noindent
	{(L2)} {\it There exists a constant $C_e\geq 1$ such that for  $i = 1,\ldots,n$,
		$$
		\mathbb P \Big(\|\bs{V}_i\|_{\infty}\vee\|\bs{Z}_i\|_{\infty} > C_e B_n\ln(p n)\Big) \leq 1/n^4.
		$$
		
	}
	\medskip
	
	\noindent
	{(L3)} {\it
		There exist constants $C_a > 0$ and $\delta\geq0$ such that for  $(x,t)\in\mathbb R \otimes (0,\infty)$,
		
		\begin{align*}
			&\mathbb P\Bigg(\s_k\Big( \frac{1}{\sqrt n}\sum_{i=1}^n \bs{Z}_i\Big)  \leq x + t \Bigg) - \mathbb P\Bigg(\s_k\Big( \frac{1}{\sqrt n}\sum_{i=1}^n \bs{Z}_i \Big) \leq x\Bigg)  \leq  C_a k \Big(t\sqrt{\ln p\vee \ln (p/t)}+\delta\Big).
		\end{align*}
	}
	\noindent

	Proposition~\ref{cor: max} gives the Kolmogorov-Smirnov bound used in the distributional comparison.
	
	\begin{proposition}[Distributional approximation]\label{cor: max}
		Assume that  Conditions (L1) to (L3) hold,
		\begin{equation*}
			\max_{1\leq j_1,j_2\leq p}\left| \frac{1}{\sqrt n}\sum_{i=1}^n\big(\Ep[V_{i j_1}V_{i j_2}] - \Ep[Z_{i j_1}Z_{i j_2}]\big) \right| \lesssim B_n\sqrt{\ln(p n)}
		\end{equation*}
		and
		\begin{equation*}
			\max_{1\leq j_1,j_2,j_3\leq p}\left| \frac{1}{\sqrt n}\sum_{i=1}^n \big(\Ep[V_{i j_1}V_{i j_2}V_{i j_3}] - \Ep[Z_{i j_1}Z_{i j_2}Z_{i j_3}]\big) \right| \lesssim B_n^2\sqrt{\ln^3(p n)}.
		\end{equation*}  Then we have
		\begin{align*}
			&\sup_{x\in\mathbb R}\bigg| \mathbb P\bigg(\s_k\Big( \frac{1}{\sqrt n}\sum_{i=1}^n \bs{V}_i \Big) \leq x\bigg) - \mathbb P\bigg(\s_k\Big( \frac{1}{\sqrt n}\sum_{i=1}^n \bs{Z}_i\Big)  \leq x\bigg) \bigg|
			\\
			&\quad \lesssim \bigg\{k^2\Big(\frac{B_n^2\ln^5(p n)}{n}\Big)^{\frac{1}{4}}+k\delta\bigg\}.
		\end{align*}
	\end{proposition}
	
	\subsection{Anti-concentration inequality }
	\label{sec:anti-concentration}
	The anti-concentration bound controls the smoothing error after the distributional comparison.
	We use the L\'{e}vy concentration function of \cite{ICL}.
	
	\begin{definition}[L\'{e}vy concentration function]  For $\epsilon>0$, the \textit{L\'{e}vy concentration function} of a real valued random variable $\xi$ is defined as
		\begin{align*}
			\mathcal{L}(\xi,\epsilon):=\sup\limits_{x \in \mathbb{R}}\mathbb P(|\xi-x|\le \epsilon).
		\end{align*}
	\end{definition}
	Proposition \ref{Proposition: anticoncentration} extends Theorem 3 of \cite{CAB} to the $k$th largest Gaussian coordinate.
	\begin{proposition}[Anti-concentration]\label{Proposition: anticoncentration}
		Let $(Y_1,\ldots, Y_p)^{\rm T}$ be a centered Gaussian random vector in $\RR^p$ with $\sigma_{j}^{2} = \mathbb E Y_{j}^{2}   > 0$ for $1 \le j \le p$, $\underline{\sigma} = \min\limits_{1 \le j \le p} \sigma_{j}$, $\bar{\sigma} = \max\limits_{1 \le j \le p} \sigma_{j}$, $a_p=\mathbb E \{\maxl_{1\le j \le p}(Y_j/\sigma_j)\}$, and $\bar{a}_p=\Ep\{\maxl_{1\le j \le p}|Y_j/\sigma_j|\}$. \\
		(i) If  $\underline{\sigma}=\overline{\sigma}=\sigma$, then for $\epsilon > 0$, $\mathcal{L}(Y_{(k)},\epsilon)\le 4k \epsilon (a_p+1)/\sigma.$\\
		(ii) If $\underline{\sigma}<\overline{\sigma}$, then for $\epsilon > 0$,
		
		$\mathcal{L}(Y_{(k)},\epsilon) \leq 8\sqrt{2}(\bar\sigma/\underline\sigma^2)k\epsilon\{\bar a_p+\sqrt{1\vee\ln(\underline\sigma/\epsilon)}\}.$
	\end{proposition}

	The following corollary provides the mean-uniform anti-concentration bound needed for the Stein-kernel comparison. It is proved below using the Gaussian anti-concentration inequality in Lemma J.3 of \cite{ICL}.
	
	\begin{corollary}[]\label{cor: anticoncentration}

		Let $\bs Z=(Z_1,\ldots,Z_p)^{\rm T}$ be a centered Gaussian random vector in $\RR^p$, where $p\geq2$, and suppose that $\min_{1\leq j\leq p}\Ep[Z_j^2]\geq\underline\sigma^2>0$. Then, for every $\bs\mu\in\RR^p$, $1\leq k\leq p$, and $\epsilon>0$,
		$$
		\sup_{x\in\RR}\Pr\{x<\mathfrak s_k(\bs Z+\bs\mu)\leq x+\epsilon\}
		\lesssim \frac{k\epsilon}{\underline\sigma}\sqrt{1\vee\ln p}.
		$$
		The implicit constant is universal. Consequently, $\mathcal L(\mathfrak s_k(\bs Z+\bs\mu),\epsilon)\lesssim k\epsilon\underline\sigma^{-1}\sqrt{1\vee\ln p}$.
	\end{corollary}

	\subsection{ Gaussian approximation  via Stein kernels}\label{sec: stein kernels}
	Stein kernels provide a direct comparison between a centered random vector and its Gaussian counterpart through covariance discrepancies.
	We recall the definition of a Stein kernel; see also \cite{ICL,K19b}.
	\begin{definition}
		Let $\bs{V}=(V_1,\ldots,V_p)^{\rm T}$ be a centered random vector in $\RR^p$. A $p\times p$ matrix-valued measurable function $\tau:\mathbb R^p\to\mathbb R^{p\times p}$
		is called a \textit{Stein kernel} for $\bs{V}$ if $\mathbb E [\tau_{j_1 j_2}(\bs{V})]<\infty$ for $1\le j_1,j_2\le p$ and
		$$
		\sum_{j=1}^p\mathbb E[\partial_j \varphi(\bs{V})V_j] = \sum_{j_1,j_2=1}^p\mathbb E [\partial_{j_1j_2}\varphi(\bs{V})\tau_{j_1j_2}(\bs{V})]
		$$
		for $\varphi\in {\mathbb C}_2(\mathbb R^p)$.
	\end{definition}
	The following proposition gives the corresponding Gaussian comparison bound.
	\begin{proposition}[Gaussian approximation via Stein kernels]\label{thm: stein kernel}
		Let $\bs{V}$ be a centered  random vector in $\mathbb R^p$ with  a  Stein kernel  $\tau$ and $ \bs{Z}$ be a centered Gaussian random vector in $\mathbb R^p$ with covariance matrix $\Delta=(\Delta_{j_1j_2})_{p\times p}$.
		Suppose that there exists a constant $C_\delta>0$ such that $\Delta_{jj}\geq C_\delta$ for  $j=1,\ldots,p$, then
		$$
		\sup_{x\in\mathbb R}\Big|\mathbb P (V_{(k)}\leq x) - \mathbb P(Z_{(k)}\leq x)\Big| \lesssim  k^{3/2} m_\delta^{1/2}\ln p,
		$$
		where $m_\delta = \mathbb E\left[\max_{1\leq j_1,j_2\leq p}|\tau_{j_1 j_2}(\bs{V}) - \Delta_{j_1 j_2}|\right]$.
	\end{proposition}

	The following corollary follows from Proposition \ref{thm: stein kernel}, Lemma 4.6 in \cite{K19b}, and the boundedness of the Stein kernel of a Beta random variable.
	\begin{corollary}[Multiplier bootstrap to Gaussian comparison]\label{coro:beta-comparison}
		Let $\bs{a}_1,\ldots,\bs{a}_n$ be vectors  in $\mathbb R^p$ with  $\bs a_i=(a_{i1},\ldots,a_{ip})^{\rm T}$ for $i=1,\ldots,n$, and
		$		\min_{1\leq j\leq p}\frac{1}{n}\sum_{i=1}^n a_{ij}^2 \geq C_b$ and $\max_{1\leq j\leq p}\frac{1}{n}\sum_{i=1}^na_{ij}^4 \leq C_B^2$		for some positive constants $C_b$ and $C_B$.
		Let $e_1,\ldots,e_n$ be {\it iid} from $N(0,1)$, and let $\varepsilon_1,\ldots,\varepsilon_n$ be {\it iid} from the standardized Beta$(\alpha,\beta)$ for some positive constants $\alpha,\beta>0$. Then
		$\mathbb E \varepsilon_i =0 , \mathbb E \varepsilon_i^2=1, i=1,\ldots,n.$
		Define $\bs{V} = \frac{1}{\sqrt n}\sum_{i=1}^n \varepsilon_i \bs{a}_i,$ $\bs{Z} = \frac{1}{\sqrt n}\sum_{i=1}^n e_i \bs{a}_i,$
		Then we have
		\begin{equation}\label{eq: beta distribution comparison}
			\begin{aligned}
				&\sup_{x\in\mathbb R}\Big|\Pr(\mathfrak s_k(\boldsymbol V)\leq x) - \Pr(\mathfrak s_k(\boldsymbol Z)\leq x)\Big|\\
				&\quad\lesssim k^{3/2}\left(\frac{C_B^2\ln^5(pn)}{n}\right)^{\frac{1}{4}}.
			\end{aligned}
		\end{equation}
	\end{corollary}
	
	\section{Kolmogorov-Smirnov distances for \texorpdfstring{$T_{(k)}$ and $T_{[k]}$}{T(k) and T[k]}}\label{sub:MR}

	We now apply the comparison tools to $T_{(k)}$ and $T_{[k]}$.
	For the Gaussian multiplier bootstrap, generate $e_1,\ldots,e_n$ independently from $N(0,1)$ and set $\bs{X}_i^*=e_i(\bs{X}_i-\overline{\boldsymbol X})$, where $\overline{\boldsymbol X}=n^{-1}\sum_{i=1}^n\bs{X}_i$.
	Define $\boldsymbol G=n^{-1/2}\sum_{i=1}^n\bs{X}_i^*$ and $G_{(k)}=\mathfrak{s}_k(\boldsymbol G)$.
	Let $\bs{Q}$ be a centered Gaussian vector in $\mathbb R^p$ with covariance matrix $\Delta_n=n^{-1}\sum_{i=1}^n\mathbb E(\bs{X}_i\bs{X}_i^{\rm T})$, and set $Q_{(k)}=\mathfrak{s}_k(\bs Q)$.
	Given $\bs X_1,\ldots,\bs X_n$, let $c_{1-\alpha}^G$ and $c_{1-\alpha}^Q$ denote the $(1-\alpha)$th quantiles of $G_{(k)}$ and $Q_{(k)}$, respectively.

	Assume $n\geq 3$ and $p\ge 3$. Let $b_1 > 0$ and $b_2 > 0$ be some constants such that $b_1\leq b_2$, and $B_n > 1$ be a sequence of constants. We allow   $B_n\to\infty$ as $n\to\infty$.
	We impose two conditions.
	
	\noindent
	{(KS1)} {\it For   $i=1,\ldots,n$ and $j = 1,\ldots,p$, $\mathbb E [\exp(|X_{i j}|/B_n)]\leq 2.$
		
	}

	\noindent
	{(KS2)} {\it For  $j=1,\ldots,p$, $  \frac{1}{n}\sum_{i=1}^n \mathbb E X_{i j}^2 \geq b_1^2 \quad\text{and}\quad \frac{1}{n}\sum_{i=1}^n \mathbb E  X_{i j}^4 \leq B_n^2 b_2^2.$
	}

	\begin{remark}\label{Comm3.1}
		
		Condition (KS1) coincides with  Condition (E.1) in \cite{CCK17} and (E.1) in \cite{DZ17}, which implies that   $X_{i j}$s are sub-exponential
		with the Orlicz norm $\|\cdot\|_{\rm e}$ bounded by $B_n$;
		see \cite{V11} for details. The first part of Condition (KS2), which was called variance lower bound condition by \cite{ICL}, requires the proper scaling of each component of  $\bs{X}_i$. This condition is essential for the application of anti-concentration inequalities (cf.  Corollary \ref{cor: anticoncentration} ).
		The second part of Condition (KS2) holds if  $X_{i j}$ are bounded  by $B_n$ for  $i$ and $j$, and $n^{-1}\sum_{i=1}^n\mathbb E X_{ij}^2 \leq b_2^2$ for  $j=1,\ldots,p$.
		Importantly, note that none of these conditions impose restrictions on the correlation matrices of $\bs{X}_i$, and therefore, our results cannot be directly derived from classical results in empirical process theory.
	\end{remark}
	
	The first theorem gives the Gaussian approximation bound for $T_{(k)}$.
	
	\begin{theorem}[Gaussian approximation for $T_{(k)}$]\label{thm: gaussian approximation main}
		Assume that Conditions (KS1) and (KS2) hold. Then
		\begin{equation*}\label{eq: gaussian bound main}
			\left|\mathbb P \left(T_{(k)} > c^Q_{1-\alpha}\right) - \alpha\right| \lesssim k^2\left(\frac{B_n^2\ln^5(p n)}{n}\right)^{\frac{1}{4}}.
		\end{equation*}
	\end{theorem}

	The following Gaussian-to-Gaussian comparison inequality is used to pass from the ideal Gaussian approximation to the multiplier-bootstrap approximation.

	\begin{proposition}[Gaussian-to-Gaussian comparison]\label{coro:g-g-comparison}
		Let $\bs{Z}_1$ and $\bs{Z}_2$ be centered Gaussian random vectors in $\mathbb R^p$ with covariance matrices $\Delta_1=(\Delta_{1,j_1j_2})_{p\times p}$ and $\Delta_2=(\Delta_{2,j_1j_2})_{p\times p}$, respectively.
		Define $M_\delta = \max_{1\leq j_1,j_2\leq p}|\Delta_{1,j_1 j_2} - \Delta_{2,j_1 j_2}|$ .
		If  $\Delta_2$ satisfies $\Delta_{2,jj}\geq c_\delta$ for  $j=1,\ldots,p$,  and $c_\delta>0$ is
		a positive constant, then
		$$
		\sup_{x\in\mathbb R}\Big|\mathbb P\big(\mathfrak{s}_k(\bs{Z}_1)  \leq x\big) - \Pr\big(\mathfrak{s}_k(\bs{Z}_2) \leq x\big)\Big|\lesssim  k^{3/2}M_\delta^{1/2}\ln p.
		$$
	\end{proposition}

	The next result gives the corresponding multiplier-bootstrap approximation for $T_{(k)}$.
	\begin{theorem}[Gaussian multiplier bootstrap approximation for $T_{(k)}$]\label{cor: rejection probabilities}
		Assume  that Conditions (KS1) and (KS2) hold. Then
		$$\left|\Pr\left(T_{(k)} > c_{1-\alpha}^G\right) - \alpha\right| \lesssim   k^2\left(\frac{B_n^2\ln^5(p n)}{n}\right)^{\frac{1}{4}}.$$
	\end{theorem}
	
	The same arguments extend to $T_{[k]}$.
	Let $\bs{Y}_i$ be the $2p$-dimensional vector obtained by concatenating $\bs{X}_i$ and $-\bs{X}_i$.
	Then $T_{[k]}=\mathfrak{s}_k(n^{-1/2}\sum_{i=1}^n\boldsymbol Y_i)$. Denote the $(1-\alpha)$th Gaussian and multiplier-bootstrap quantiles of $T_{[k]}$ by $c_{1-\alpha}^{[Q]}$ and $c_{1-\alpha}^{[G]}$, respectively.

	\begin{theorem}[Gaussian approximation for $T_{[k]}$]\label{thm2: gaussian approximation main}
		Assume that Conditions (KS1) and (KS2) hold. Then
		\begin{equation*}\label{eq: gaussian bound abs}
			\left|\mathbb P \left(T_{[k]} > c^{[Q]}_{1-\alpha}\right) - \alpha\right| \lesssim k^2\left(\frac{B_n^2\ln^5(p n)}{n}\right)^{\frac{1}{4}}.
		\end{equation*}
	\end{theorem}

	\begin{theorem}[Gaussian multiplier bootstrap approximation for $T_{[k]}$]\label{cor2: rejection probabilities}
		Assume  that Conditions (KS1) and (KS2) hold. Then
		$$\left|\Pr\left(T_{[k]} > c_{1-\alpha}^{[G]}\right) - \alpha\right| \lesssim   k^2\left(\frac{B_n^2\ln^5(p n)}{n}\right)^{\frac{1}{4}}.$$
	\end{theorem}

	\section{Kolmogorov-Smirnov distances for \texorpdfstring{$T_\psi$}{T psi}}
	Using the notation above, define the centered Gaussian multiplier sum
	$$\boldsymbol G = (G_1,\ldots, G_p)^{\rm T} = \frac{1}{\sqrt{n}} \sum_{i=1}^n \bs{X}_i^*.$$
	Rank $|G_{1}|,\ldots,|G_{p}|$ in decreasing order and denote the ordered values by $G_{[1]},\ldots,G_{[p]}$.
	Let $T_\psi^{\boldsymbol{G}}=\psi(G_{[1]},\ldots,G_{[k]})$ and define
	$$\rho_1=\sup\limits_{t\in \mathbb{R}}\big| \mathbb{P}\big(T_\psi \leq t\big) - \mathbb{P}\big(T_\psi^{\boldsymbol{G}}  \leq t \mid \boldsymbol X_1, \ldots, \boldsymbol X_n \big) \big|.$$

	Similarly, let $\bs{Q}=(Q_1,\ldots,Q_p)^{\rm T}$ be a centered Gaussian vector with covariance matrix $\Delta_n=n^{-1}\sum_{i=1}^n\mathbb E(\bs{X}_i\bs{X}_i^{\rm T})$.
	Rank $|Q_1|,\ldots,|Q_p|$ in decreasing order and denote the ordered values by $Q_{[1]},\ldots,Q_{[p]}$.
	Let $T_\psi^{\boldsymbol{Q}}=\psi(Q_{[1]},\ldots,Q_{[k]})$ and define
	$$\rho_2=\sup\limits_{t\in \mathbb{R}}\big| \mathbb{P}\big(T_\psi \leq t\big) - \mathbb{P}\big(T_\psi^{\boldsymbol{Q}}  \leq t  \big) \big|.$$
	We impose the following conditions.

	\noindent{(KS3)}  {\it $  \frac{1}{n} \sum_{i=1}^n \mathbb{E}[(\boldsymbol v^{\rm T} \boldsymbol X_i)^2] \geq b_1^2$ for $\boldsymbol v \in \mathbb{R}^{p}$ with $\|\boldsymbol v\|_0 \leq k$ and $\|\boldsymbol v\|_2 = 1$.}
	
	\noindent{(KS4)} {\it $\psi\big(u_1,\ldots,u_k\big)\geq \psi\big(v_1,\ldots,v_k\big)$ when $u_l\geq v_l\geq0$ for $l = 1,\ldots,k$.}
	
	\noindent{(KS5)} {\it $\psi\big(u_{(1)},\ldots,u_{(k)}\big) \geq \psi\big(u_1,\ldots,u_k\big)$ when $u_l\geq0$ for $l = 1,\ldots,k$.}
	
	\noindent{(KS6)} {\it $\big\{(u_1, \ldots, u_k)\in \mathbb{R}^k \mid \psi(u_1,\ldots, u_k)\leq t\big\}$ is closed and convex for $t\in \mathbb{R}$.}
	
	\begin{remark}
		
		Condition (KS3) coincides with Condition (M.1$^{\prime \prime}$) in \cite{CCK17}. 
		Conditions (KS4) and (KS5) imply that $\psi$ is monotone, whereas Condition (KS6) ensures that $\psi$ is Borel measurable and quasi-convex.
		
	\end{remark}

	The following two theorems give the Gaussian and multiplier-bootstrap bounds for $T_{\psi}$.
	
	\begin{theorem}[Gaussian approximation for $T_{\psi}$]\label{thm5.1} Suppose that Conditions (KS1)-(KS6) are satisfied. When $n$ is sufficiently large, we have
		
		$$ {\rho_2} \lesssim k^{\frac{7}{2}} \Big( \frac{B_n^2 \ln^5(\gamma k^{\frac{1}{4}}pn^{\frac{13}{8}}\ln^{\frac{1}{2}}(2pn^{\frac{5}{4}})\ln[k^{\frac{1}{2}}n^{\frac{5}{4}}\ln(2pn^{\frac{5}{4}})])}{n} \Big)^{\frac{1}{4}}, $$
		where $\gamma$ is a constant.
		
	\end{theorem}

	\begin{theorem}[Gaussian multiplier bootstrap approximation for $T_{\psi}$]\label{thm5.2} Suppose that Conditions (KS1)-(KS6) are satisfied. When $n$ is sufficiently large, we have
		$$ {\rho_1} \lesssim k^{\frac{7}{2}} \Big( \frac{B_n^2 \ln^5(\gamma k^{\frac{1}{4}}pn^{\frac{13}{8}}\ln^{\frac{1}{2}}(2pn^{\frac{5}{4}})\ln[k^{\frac{1}{2}}n^{\frac{5}{4}}\ln(2pn^{\frac{5}{4}})])}{n} \Big)^{\frac{1}{4}}, $$
		where $\gamma$ is a constant.
	\end{theorem}

	\begin{corollary}[Gaussian approximation and Gaussian multiplier bootstrap approximation for $T_{\psi}$]\label{cor5.1} Assume that Conditions (KS1)-(KS3) hold. When $n$ is sufficiently large, we have the following results.
		
		\noindent(i) If $\psi\big(u_1, \ldots, u_k\big)=\sum_{l=1}^{k}a_lu_l$, where $a_1 \geq \ldots \geq a_k \geq 0$ are constants,
		$$\rho_1\vee \rho_2 \lesssim k \Big( \frac{B_n^2 \ln^5(2^kp^kn)}{n} \Big)^{\frac{1}{4}};$$
		
		\noindent(ii)  If $\psi\big(u_1, \ldots, u_k\big)=\sum_{l=1}^{k}a_l\exp(u_l)$, where $a_1 \geq \ldots \geq a_k \geq 0$ are constants,
		$$\rho_1\vee \rho_2 \lesssim k^{\frac{7}{2}} \Big( \frac{B_n^2 \ln^5(\gamma k^{\frac{1}{4}}pn^{\frac{13}{8}}\ln^{\frac{1}{2}}(2pn^{\frac{5}{4}})\ln[k^{\frac{1}{2}}n^{\frac{5}{4}}\ln(2pn^{\frac{5}{4}})])}{n} \Big)^{\frac{1}{4}};$$
		
		\noindent(iii) If $\psi\big(u_1, \ldots, u_k\big)=\sum_{l=1}^{k}a_lu_l^2$, where $a_1 \geq \ldots \geq a_k \geq 0$ are constants,
		$$\rho_1\vee \rho_2 \lesssim k^{\frac{7}{2}} \Big( \frac{B_n^2 \ln^5(\gamma k^{\frac{1}{4}}pn^{\frac{13}{8}}\ln^{\frac{1}{2}}(2pn^{\frac{5}{4}})\ln[k^{\frac{1}{2}}n^{\frac{5}{4}}\ln(2pn^{\frac{5}{4}})])}{n} \Big)^{\frac{1}{4}},$$
		in which $\gamma$ is a constant.
	\end{corollary}
	
	 In view of Theorems 5.1 and 5.2 and Corollary 5.1,  the dimension $p$ is allowed to grow as fast as  $\ln p = o(n^{1/5})$ to ensure convergence, assuming $B_n = O(1)$.

	\section{Combining several sets of top order statistics}
	Denote by $\psi_l$ an $l$-dimensional function on $(u_1,\ldots,u_l)^{\rm T} \in\mathbb R^l$ for $l = 1,\ldots,k$ which satisfies Conditions (KS4)-(KS6).
	We propose to combine some large order  statistics as
	$$T_F=\min\big\{1- F_{l_1}\big( T_{\psi_{l_1}}\big),\ldots,1- F_{l_s}\big(T_{\psi_{l_s}}\big)\big\}, $$
	where  $ F_l(t)$  is the cumulative distribution function  of $T_{\psi_l}^{\bs G} $ given $\boldsymbol X_1, \ldots, \boldsymbol X_n$
	for $l=l_1, \ldots, l_s$ with $1\leq l_1 < \ldots < l_s=k$.
	Based on Theorem 5.1, given $\bs{X}_{1},\ldots,\bs{X}_n$, a conventional approach  approximating the tail probability  of $T_F$ is to employ a two-layer resampling procedure, where the first layer is for the approximation of $ F_{l_1},\ldots, F_{l_s}$
	and the second one is for the tail probability.
	This nested procedure is computationally intensive.
	To alleviate this burden, we propose a more efficient one-layer Gaussian multiplier bootstrap procedure  described below.
	
	\vspace{2ex}

	\noindent{\bf One-layer Gaussian multiplier bootstrap procedure for $T_F$.} \\
	{\bf\textit{Step 1.}} (Generate Gaussian multiplier samples of $\bs T$). Set a sufficiently large $B>0$. Given $\boldsymbol X_1, \ldots, \boldsymbol X_n$, calculate $\overline{\boldsymbol X} = n^{-1}\sum_{i=1}^n \boldsymbol X_i$ and generate $B$ sets of Gaussian multiplier bootstrap samples   as
	$$
	\boldsymbol G^{(r)} = (G_1^{(r)},\ldots, G_p^{(r)})^{\rm T}
	= \frac{1}{\sqrt{n}} \sum_{i=1}^n (\boldsymbol X_i - \overline{\boldsymbol X}) e_i^{(r)}, r=1,\ldots,B,
	$$
	where  $e_1^{(1)},\ldots,e_n^{(1)},\ldots, e_1^{(B)},\ldots,e_n^{(B)}$ are {\it iid} from $N(0,1)$.
	For $r=1,\ldots,B$, rank  $|G^{(r)}_1|,\ldots,|G^{(r)}_p|$ in a decreasing order and denote them by
	$G_{[1]}^{(r)},\ldots,G_{[p]}^{(r)}$.
	Then  calculate $T_{\psi_{l}}^{\bs G^{(r)}}=\psi_l\bigl(G_{[1]}^{(r)},\ldots,G_{[l]}^{(r)}\bigr)$ for $l=l_1, \ldots, l_s$ and $r=1,\ldots,B$.  \\
	{\bf\textit{Step 2.}} (Calculate the observed value of $T_F$). Calculate the observed value of $T_F$ as
	$$t_0=\min\Big\{ \frac{1}{B}\sum_{r=1}^{B}\mathds{1}\big\{T_{\psi_{l_1}}^{\bs G^{(r)}} > T_{\psi_{l_1}}\big\},\ldots, \frac{1}{B}\sum_{r=1}^{B}\mathds{1}\big\{T_{\psi_{l_s}}^{\bs G^{(r)}} > T_{\psi_{l_s}}\big\}\Big\}. $$ \\
	{\bf\textit{Step 3.}} (Calculate the empirical values of $T_F$). For $r=1,\ldots,B$, calculate
	$$t_r
	=\min\Big\{\frac{1}{B-1}
	\sum_{\substack{h=1 \\ h\neq r}}^B\mathds{1}\big\{T_{\psi_{l_1}}^{\bs G^{(h)}} > T_{\psi_{l_1}}^{\bs G^{(r)}}\big\},\ldots,\frac{1}{B-1}\sum_{\substack{h=1 \\ h\neq r}}^B\mathds{1}\big\{T_{\psi_{l_s}}^{\bs G^{(h)}} > T_{\psi_{l_s}}^{\bs G^{(r)}}\big\}
	\Big\}. $$    \\
	{\bf\textit{Step 4.}} (Calculate the tail probability). The tail probability of $T_F$ is approximated by
	$$\frac{1}{B}\sum_{r=1}^B \mathds{1}\{t_r< t_0\}. $$ \\
	
	In what follows, we analyze the consistency of the proposed one-layer bootstrap procedure.
	
	\begin{theorem}[Gaussian multiplier bootstrap approximation for $T_{F}$]\label{thm6.1} Suppose that Conditions (KS1)-(KS3) are satisfied. When $n$ is sufficiently large, we have
		$$\lim_{B\rightarrow \infty }\Big|\frac{1}{B}\sum_{r=1}^B\mathds{1}{\{t_r<t\}}-\mathbb{P}(T_F< t)\Big|\lesssim k^{\frac{7}{2}} \Big( \frac{B_n^2 \ln^5(\gamma k^{\frac{1}{4}}pn^{\frac{13}{8}}\ln^{\frac{1}{2}}(2pn^{\frac{5}{4}})\ln[k^{\frac{1}{2}}n^{\frac{5}{4}}\ln(2pn^{\frac{5}{4}})])}{n} \Big)^{\frac{1}{4}} $$
		almost surely for $t\in \mathbb{R}$, where $\gamma$ is a constant.
	\end{theorem}

	\section{Numerical Studies}

	\subsection{Simulation studies} \label{simu}

	We evaluate the accuracy of Gaussian multiplier bootstrap in approximating the tail probabilities of  $T_{(k)}$, $T_{\psi_k}$, and $T_F$.
	The results for $T_{[k]}$ are qualitatively similar to those for $T_{(k)}$ and are therefore omitted for brevity.
	We generate independent $\bs X_1,\ldots,\bs X_{1000}$ from a $300$-dimensional $t$-distribution
	$t_{10}({\bf 0}_{300}, \boldsymbol\Theta)$ with $10$ degrees of freedom and covariance matrix
	$\boldsymbol\Theta = (\theta_{ij})_{300\times 300}$, where $\theta_{ij} = \rho^{|i-j|}$ with $\rho=0.2$ or $0.8$.
	We generate $2000$ Gaussian multiplier bootstrap samples. 
	Based on $2000$ Monte Carlo replicates, Q--Q plots are employed to compare the empirical tail probabilities with their theoretical counterparts.
	For the order statistic $T_{(k)}$, we consider $k=1,\ldots,6$.
	For $k=1,\ldots,5$, we consider the sum-of-squares function and construct  $T_{\psi_k} = \sum_{l=1}^k T_{[l]}^2$.
	Moreover, we select  the indices 1,3, and 5, and construct $T_F=\min\big\{1- F_{1}\big( T_{\psi_1}\big),1- F_{3}\big(T_{\psi_{3}}\big),1- F_{5}\big(T_{\psi_{5}}\big)\big\}. $
	The  Gaussian multiplier bootstrap procedure to calculate the tail probability for $T_{(k)}$ is detailed below. The procedure for $T_{\psi_k}$ follows similarly and is therefore omitted, while the procedure for $T_F$ is presented in Section 6.
	
	\noindent{\bf Gaussian multiplier bootstrap procedure for $T_{(k)}$  } \\
	{\bf\textit{Step 1.}} (Calculate the observed value $T_{(k)}$). Denote $\boldsymbol T = (T_1, \ldots, T_{300})^{\rm T} = \frac{1}{\sqrt{1000}} \sum_{i=1}^{1000} \boldsymbol X_i$.
	Rank $T_1,\ldots,T_{300}$ in a decreasing order and denote them by $T_{(1)},\ldots,T_{(300)}$.\\ 
	{\bf\textit{Step 2.}} (Generate Gaussian multiplier samples of $T_{(k)}$). Given $\boldsymbol X_1, \ldots, \boldsymbol X_{1000}$, calculate $\overline{\boldsymbol X} = \frac{1}{1000}\sum_{i=1}^{1000} \boldsymbol X_i$, and generate $2000$ Gaussian multiplier bootstrap samples as
	\[
	\boldsymbol G^{(r)} = (G_1^{(r)},\ldots, G_{300}^{(r)})^{\rm T}
	= \frac{1}{\sqrt{1000}} \sum_{i=1}^{1000} (\boldsymbol X_i - \overline{\boldsymbol X}) e_i^{(r)}, r=1,\ldots,2000,
	\]
	where  $e_1^{(1)},\ldots,e_{1000}^{(1)},\ldots, e_1^{(2000)},\ldots,e_{1000}^{(2000)}$ are {\it iid} from $N(0,1)$.
	For $r=1,\ldots,2000$, rank $G_1^{(r)},\ldots, G_{300}^{(r)}$ in a decreasing order and denote them by
	$G_{(1)}^{(r)},\ldots,G_{(300)}^{(r)}$.\\ 
	{\bf\textit{Step 3.}} (Calculate the tail probability of $T_{(k)}$). For $k=1,\ldots, 6$, the tail probability is approximated by
	$$\frac{1}{2000}\sum_{r=1}^{2000} \mathds{1}\{G_{(k)}^{(r)}> T_{(k)}\}. $$ \\
	
	Figures \ref{tab:fig1} to \ref{tab:fig3} display the theoretical tail probabilities against the sample tail probabilities based on the  Gaussian multiplier bootstrap procedure  in the  considered scenarios.
	They show that the tail probabilities derived from the Gaussian multiplier bootstrap procedure asymptotically follow the uniform distribution.
	The plots indicate that the quantiles of $T_{(k)}$, $T_{\psi_k}$ and $T_F$ are well approximated by the Gaussian multiplier bootstrap across the two correlation settings.

	\begin{figure}
		\begin{center}
			\includegraphics[scale=0.6]{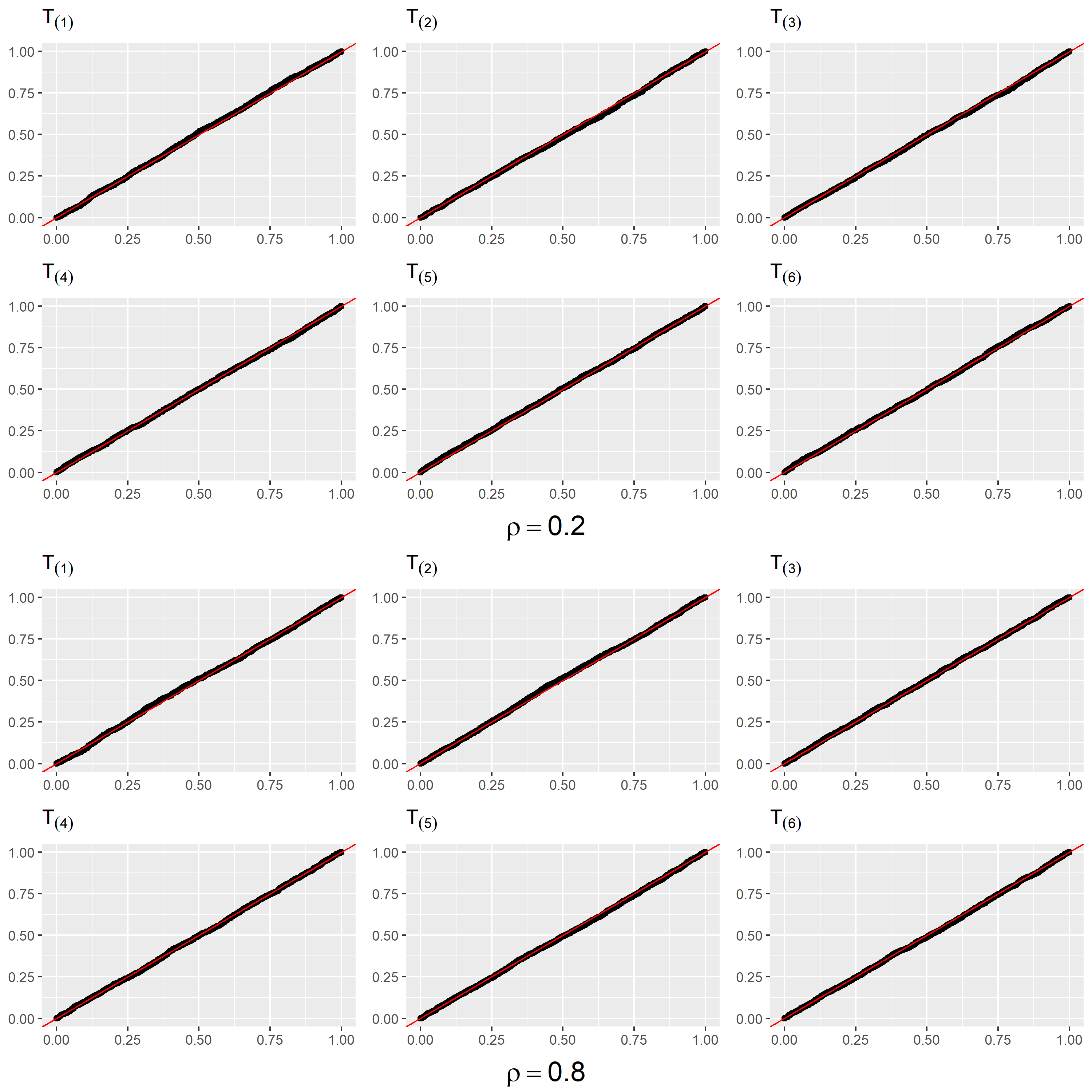}
		\end{center}\vspace{-0.5cm}
		\caption{The theoretical tail probabilities against the sample  tail probabilities based on the Gaussian multiplier bootstrap procedure for $T_{(k)}$ when $k= 1,\ldots,6 $ with $\rho=0.2$ or $0.8$.}
		\label{tab:fig1}
	\end{figure}
	
	\begin{figure}
		\begin{center}
			\includegraphics[scale=0.6]{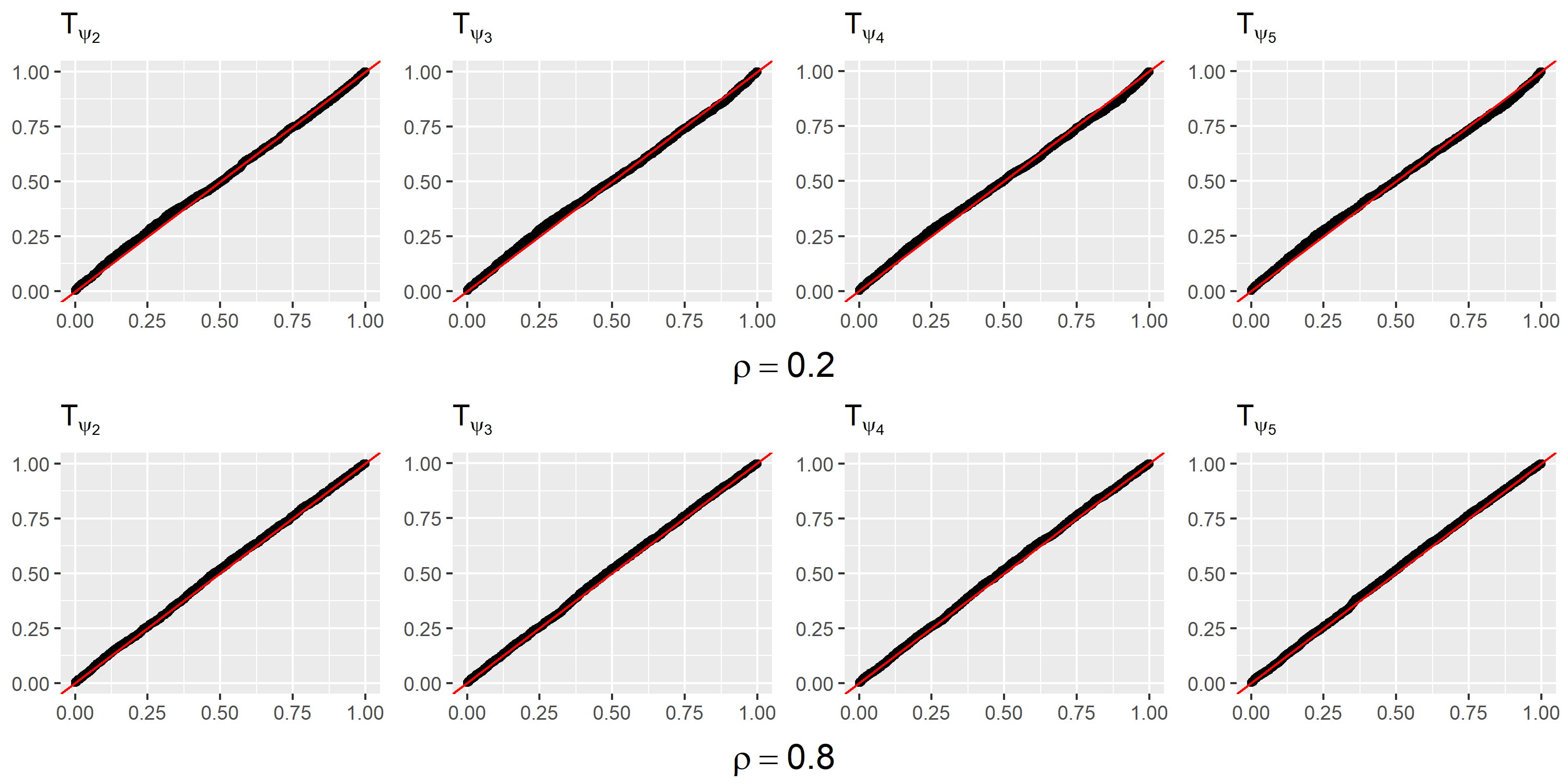}
		\end{center}\vspace{-0.5cm}
		\caption{ The theoretical tail probabilities against the sample tail probabilities based on the Gaussian multiplier bootstrap procedure for $T_{\psi_{k}}$ when $k= 2,\ldots,5 $ with $\rho=0.2$ or $0.8$.}
		\label{tab:fig2}
	\end{figure}

	\begin{figure}
		\begin{center}
			\includegraphics[scale=0.4]{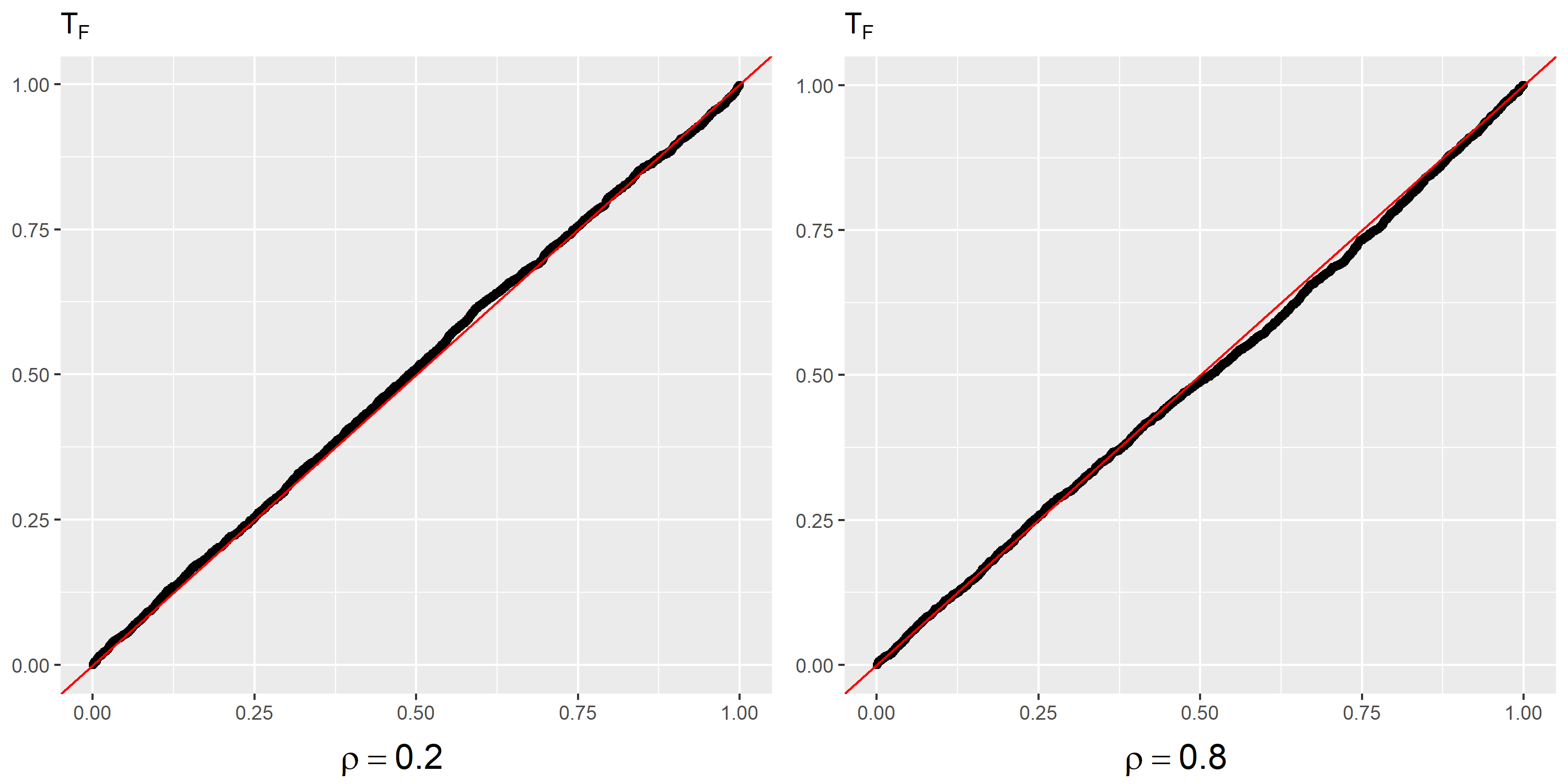}
		\end{center}\vspace{-0.5cm}
		\caption{ The theoretical tail probabilities against the sample tail probabilities based on the Gaussian multiplier bootstrap procedure for $T_F$  with $\rho=0.2$ or $0.8$.}
		\label{tab:fig3}
	\end{figure}

	\subsection{ Application to the urban traffic speed data}
	
	We apply the proposed Gaussian multiplier bootstrap procedures to  the  dataset (\url{https://github.com/sysuits}) collected in Guangzhou, China,   comprising urban traffic speed data from 214 anonymous road segments, including urban expressways and arterial roads, over a two-month period (August 1, 2016, to September 30, 2016) at 10-minute intervals.
	Based on its inherent spatial and temporal attributes, the data can be organized into a third-order tensor with dimensions corresponding to road segments, days, and time windows.
	
	The dataset initially exhibits a missing data rate of 1.29\%; these gaps are filled using mean imputation.
	For our analysis, we utilize traffic speed data recorded at 8:00 A.M. and 18:00 P.M. spanning 61 days from 214 road segments.
	Our objective is to assess the null hypothesis $H_0$  that the variation in urban traffic speed among different locations is negligible.
	This comparison provides evidence on spatial differences in congestion across Guangzhou between the two time points.
	We consider $T_{(k)}$, $T_{\psi_k} = \sum_{l=1}^k T_{[l]}^2$, and $T_F=\min\big\{1- F_{1}\big( T_{\psi_1}\big),1- F_{3}\big(T_{\psi_{3}}\big),1- F_{5}\big(T_{\psi_{5}}\big)\big\}, $
	where $F_k$ denotes the conditional CDF of $T^{\bs G}_{\psi_k}$ given the data for $k=1,\ldots,5$. All tests reject $H_0$ with p-values smaller than $0.001$.
	This finding indicates that the urban traffic speed significantly differs between 8:00 A.M. and 18:00 P.M. across the analyzed locations.
	
	\section{A Concluding Remark}
	We studied Gaussian and Gaussian multiplier bootstrap approximations for the $k$th largest coordinate and for functions of the top $k$ ordered coordinates of sums of independent random vectors.
	The results allow the dimension to grow exponentially with the sample size under the stated moment conditions and show that multiplier-bootstrap critical values attain the same order of approximation error as their Gaussian counterparts.
	Future work could explore extensions to dependent data, relax moment constraints via truncation, or adapt the methodology to sparse high-dimensional models.

	\section{Technique details}
	
	\subsection{Proof of Lemma \ref{Derivatives of m}}
	By applying the chain rule and  Lemma A.2 of \cite{GAR},
	we establish the following auxiliary lemma for the proof of Lemma  \ref{Derivatives of m}.

	\begin{lemma}\label{chain rule}
		Let  $\bs{g}: \mathbb{R}^{d} \to \mathbb{R}^{q},\boldsymbol{g}(\bs{x})=\big(g_{1}(\bs{x}), g_{2}(\bs{x}), \ldots, g_{q}(\bs{x})\big)^{\rm T}$  be a function satisfying
		$$\sum_{j=1}^{d}\left|\partial_{j} g_{i}\right| \leq C_{g,1}, \sum_{j_1, j_2=1}^{d}\left|\partial_{j_2} \partial_{j_1} g_{i}\right| \leq C_{g,2}, \sum_{j_1,j_2, j_3=1}^{d}\left|\partial_{j_3} \partial_{j_2} \partial_{j_1} g_{i}\right| \leq C_{g,3},$$
		$$
		\sum_{j_1, j_2, j_3,j_4=1}^{d}\left|\partial_{j_4}\partial_{j_3} \partial_{j_2} \partial_{j_1} g_i\right| \leq C_{g,4}, ~\text{and} ~
		\sum_{j_1, j_2, j_3,j_4,j_5=1}^{d}\left|\partial_{j_5} \partial_{j_4}\partial_{j_3} \partial_{j_2} \partial_{j_1} g_i\right| \leq C_{g,5},
		$$
		uniformly for  $\bs{x}\in\RR^d$ and $i=1,  \ldots, q$, where integers  $d, q \geq 1$  and  $C_{g,1}, C_{g,2}, C_{g,3}, C_{g,4},$ and $ C_{g,5}$  are positive constants. Let  $f: \mathbb{R}^{q} \to \mathbb{R}$  be a five-times continuously differentiable function satisfying
		$$\sum_{j=1}^{q}\left|\partial_{j} f\right| \leq   C_{f,1}, \sum_{j_1, j_2=1}^{q}\left|\partial_{j_2} \partial_{j_1} f\right| \leq C_{f,2},
		\sum_{j_1, j_2, j_3=1}^{q}\left|\partial_{j_3} \partial_{j_2} \partial_{j_1} f\right| \leq C_{f,3}, $$
		$$
		\sum_{j_1, j_2, j_3,j_4=1}^{q}\left|\partial_{j_4}\partial_{j_3} \partial_{j_2} \partial_{j_1} f\right| \leq C_{f,4}, ~\text{and} ~
		\sum_{j_1, j_2, j_3,j_4,j_5=1}^{q}\left|\partial_{j_5} \partial_{j_4}\partial_{j_3} \partial_{j_2} \partial_{j_1} f\right| \leq C_{f,5},
		$$
		where  $C_{f,1}, C_{f,2}, C_{f,3}, C_{f,4},$  and $ C_{f,5} $ are positive constants. Then we have
		$$\sum_{j=1}^{d}\left|\partial_{j}(f \circ \bs{g})\right| \leq C_{g,1} C_{f,1}, \quad \sum_{j_1, j_2=1}^{d}\left|\partial_{j_2} \partial_{j_1}(f \circ \bs{g})\right| \leq C_{g,1}^{2} C_{f,2}+C_{g,2} C_{f,1},$$
		$$\sum_{j_1, j_2, j_3,j_4=1}^{d}\left|\partial_{j_3} \partial_{j_2} \partial_{j_1}(f \circ \bs{g})\right| \leq C_{g,1}^{3} C_{f,3}+3 C_{g,1} C_{g,2} C_{f,2}+C_{g,3} C_{f,1},$$
		$$\sum_{j_1, j_2, j_3=1}^{d}\left|\partial_{j_4}\partial_{j_3} \partial_{j_2} \partial_{j_1}(f \circ \bs{g})\right| \leq C_{g,1}^4C_{f,4}+3C_{g,1}^2C_{g,2}C_{f,3}+3C_{g,1}^2C_{g,2}C_{f,2}+4C_{g,1}C_{g,3}C_{f,2}+3C_{g,2}^2C_{f,2}+C_{g,4}C_{f,1},$$
		and
		\begin{align*}
			\sum_{j_1, j_2, j_3, j_4,j_5=1}^{d}\left|\partial_{j_5}\partial_{j_4}\partial_{j_3} \partial_{j_2} \partial_{j_1}(f \circ \bs{g})\right| &\leq C_{g,1}^5C_{f,5}+7C_{g,1}^3C_{g,2}C_{f,4}+3C_{g,1}^3C_{g,2}C_{f,3}+9C_{g,1}C_{g,2}^2C_{f,3}+7C_{g,1}^2C_{g,3}C_{f,3}\\
			&+6C_{g,1}C_{g,2}^2C_{f,2}+3C_{g,1}^2C_{g,3}C_{f,2}+5C_{g,1}C_{g,4}C_{f,2}+10C_{g,2}C_{g,3}C_{f,2}+C_{g,5}C_{f,1}.
		\end{align*}
	\end{lemma}

	\noindent{ \bf Proof of Lemma \ref{Derivatives of m}} 
	We first bound the derivatives of $g$, $\mathfrak{f}_{N}$, and $\mathbf{H}$.
	Then, we apply  the chain rule to obtain claims (i) and (ii).
	Finally, we prove the claim (iii) by properties of the smooth max function.	
	
	Denote $d_r= (k^{-1}\phi)^r C_g, r=1,\ldots, 5$, we have
	\begin{align}\label{M_t}
		\supl_{x\in \RR} |\partial^{r}g(x)|\le d_r ~~\text{and}~~ d_{r_1}\beta^{r-r_1}\le d_1\beta^{r-1}= C_g k^{-1}\phi^r\ln^{r-1} p, \ r_1=1,\ldots,r.
	\end{align}
	Based on the chain rule, we obtain
	$$\partial_j \mathfrak{f}_N\ge 0,\quad \sum_{j=1}^{p}\partial_j \mathfrak{f}_N=1,\quad \sum_{j_1,j_2=1}^{p}|\partial_{j_1} \partial_{j_2}\mathfrak{f}_N| \le 2\beta,~\text{and} ~\sum_{j_1,j_2,j_3=1}^{p}|\partial_{j_1} \partial_{j_2}\partial_{j_3}\mathfrak{f}_N|\le 6\beta^2, $$
	for $1 \le j_1, j_2, j_3 \le  d$ (see also Lemma A.2 of \cite{GAR}).
	Based on some algebras, we have
	\begin{align*}
		\sum_{j=1}^{p}|\partial_j g_{k,A}| =1, \quad \sum_{j_1,j_2=1}^{p}|\partial_{j_2}\partial_{j_1} g_{k,A}| \les \beta, \quad \sum_{j_1,j_2,j_3=1}^{p}|\partial_{j_3} \partial_{j_2} \partial_{j_1}g_{k,A}| \les \beta^2,\\
		\sum_{j_1,j_2,j_3,j_4=1}^{p}| \partial_{j_4}\partial_{j_3} \partial_{j_2} \partial_{j_1} g_{k,A}| \les \beta^3,  ~\text{and} ~ \sum_{j_1,j_2,j_3,j_4,j_4=1}^{p}|\partial_{j_5} \partial_{j_4}\partial_{j_3} \partial_{j_2} \partial_{j_1} g_{k,A}| \les \beta^4,
	\end{align*}
	for  $A\in \mathscr{A}_k$.
	Recalling \eqref{M_t} and  $\beta=\phi\ln p\ge 1$, and applying Lemma \ref{chain rule} to $g\circ \mathfrak{f}_{N}$,  we have
	\begin{align*}
		&\qquad	\qquad	\sum_{j=1}^{N}|\partial_j (g\circ \mathfrak{f}_{N})|\les d_1, \quad \sum_{j_1,j_2=1}^{N}|\partial_{j_2}\partial_{j_1} (g\circ \mathfrak{f}_{N})| \les d_2+d_1\beta\les d_1\beta,  \\
		&\sum_{j_1,j_2,j_3=1}^{N}|\partial_{j_3}\partial_{j_2}\partial_{j_1} (g\circ \mathfrak{f}_{N})| \les d_3+d_2\beta+d_1\beta^2\les d_1\beta^2, \sum_{j_1,j_2,j_3,j_4=1}^{N}| \partial_{j_4}\partial_{j_3} \partial_{j_2} \partial_{j_1} (g\circ \mathfrak{f}_{N})| \les d_1\beta^3,
	\end{align*}
	and
	\begin{align*}	
		\sum_{j_1,j_2,j_3,j_4,j_5=1}^{N}|\partial_{j_5} \partial_{j_4}\partial_{j_3} \partial_{j_2} \partial_{j_1} (g\circ \mathfrak{f}_{N})| \les d_1\beta^4.
	\end{align*}
	Applying Lemma \ref{chain rule} to $(g\circ \mathfrak{f}_{N})\circ \mathbf{H}$ once again, we obtain claims (i) and (ii).

	Taking the partial derivatives of $\mathfrak{q}_k$ and $g_{k,A}$ with respect to $x_j$, we have
	\begin{align*}
		\pi_j:=\partial_j \mathfrak{q}_k=\frac{\partial_j\tilde{g}_k}{\beta \tilde{g}_k}=\frac{1}{\tilde{g}_k}\sum_{A\in \mathscr{A}_k(j)}\exp(\beta g_{k, A})\partial_jg_{k,A},
	\end{align*}
	and
	\begin{align*}
		\sigma_{j,A}(\bs{x}):=\partial_jg_{k,A}(\bs{x})=\frac{\exp(-\beta x_j )}{\sum_{s\in A}\exp( -\beta x_s) }\mathds{1}\{j\in A\}.
	\end{align*}

	For claim (iii), by the properties of smooth max function (see Lemma A.2 and Lemma A.6 in \cite{GAR}), we have $\sigma_{j,A}(\bs{z}+\bs{w})\le {\rm e}^2\sigma_{j,A}(\bs{w})$ for  $\bs{z}\in\mathbb R^p$ and $\bs{w}\in\mathbb R^p$, where ${\rm e}$ is the base of the natural logarithm. Moreover,
	\begin{align*}
		\frac{\exp(\beta g_{k,A}(\bs{z}+\bs{w}))}{\exp(\beta g_{k,A}(\bs{z}))}
		&=\exp\left\lbrace \beta\left( g_{k,A}(\bs{z}+\bs{w})-g_{k,A}(\bs{z})\right)  \right\rbrace  \\
		&=\exp\left( \beta\cdot\frac{1}{\beta}\ln\left( \frac{\sum_{s\in A}\exp(-\beta z_s)}{\sum_{s\in A} \exp(-\beta z_s-\beta w_s)}\right) \right) \\
		&=\frac{\sum_{s\in A}\exp(-\beta z_s)}{\sum_{s\in A} \exp(-\beta z_s-\beta w_s)}\in[{\rm e}^{-1},{\rm e}].
	\end{align*}
	Therefore, $\pi_j(\bs{z}+\bs{w})\les \pi_j(\bs{z})$. It can be easy to  show that $U_{j_1j_2},\ldots, U_{j_1j_2j_3j_4j_5}$ are finite sums of  products of terms such as $\exp(\beta g_A), \sigma_{j,A}, \pi_j$, and the claim (iii) follows.  \rulex
	
	\vspace{2ex}

	
	
	\subsection{Proof of Proposition \ref{cor: max}}
	
	Using the iterative randomized Lindeberg method developed by \cite{ICL}, we  establish  Lemma \ref{lem: main}, Lemma \ref{lem: closing}, and Corollary \ref{cor: main}, which collectively imply Proposition \ref{cor: max}.
	
	We first introduce some definitions.
	Note that 	  $\bs{V}_1,\ldots,\bs{V}_n, \bs{Z}_1,\ldots,\bs{Z}_n$
	are independent $p$-dimensional  random vectors with components satisfying $\mathbb E V_{i j} = \mathbb E Z_{i j} = 0$ for  $i = 1,\ldots,n$ and $j = 1,\ldots,p$.
	For   $\bs{\epsilon}=(\epsilon_1,\ldots,\epsilon_n)^{\rm T}\in\{0,1\}^n= \{0,1\}\times \cdots \times \{0,1\}$, where  $\times$ denotes the  Cartesian product, define
	\begin{equation*}\label{eq: rho eps definition}
		\rho_{\bs\epsilon} = \sup_{x\in\mathbb R}
		\big|\Pr \big(\s_k ( \bs{S}_{n,\bs\epsilon}^V )  \leq x\big) - \Pr \big(\s_k(\bs{S}_n^Z) \leq x \big) \big|,
	\end{equation*}
	where
	$$
	\bs{S}_{n,\bs\epsilon}^V =\frac{1}{\sqrt n}\sum_{i=1}^n \big(\epsilon_i\bs{V}_i + (1 - \epsilon_i)\bs{Z}_i\big) \quad\text{and}\quad \bs{S}_n^Z =\frac{1}{\sqrt n}\sum_{i=1}^n \bs{Z}_i.
	$$
	We replace $\bs{\epsilon}$ with a certain sequence of random vectors $\bs{\epsilon}^0,\ldots,\bs{\epsilon}^D \in \{ 0,1 \}^n$,  independent of $\bs{V}_1,\ldots,\bs{V}_n,\bs{Z}_1,\ldots,\bs{Z}_n$, and obtain the  recursive bounds for $\rho_{\bs{\epsilon}^d},$ $d=0,\ldots,D$.
	The construction and properties of $\bs{\epsilon}^0,\ldots,\bs{\epsilon}^D \in \{ 0,1 \}^n$ were demonstrated in Section 2 in \cite{ICL}.
	Let $C_{n,1,d}>0$ and $C_{n,2,d}>0$ be some constants for  $n\geq 1$ and $d = 0,\ldots,D$, and define the event $\mathcal A_d$ by
	\begin{align*}
		\mathcal A_d&=
		\bigg\{\max_{1\leq j_1,j_2\leq p}\left| \frac{1}{\sqrt n}\sum_{i=1}^n \epsilon^d_i\big(\Ep[V_{ij_1}V_{ij_2}]- \Ep[Z_{ij_1 }Z_{ij_2}]   \big) \right| \leq C_{n,1,d}\bigg\} \\
		&\quad \bigcap
		\bigg\{ \max_{1\leq j_1,j_2,j_3\leq p}\left| \frac{1}{\sqrt n}\sum_{i=1}^n \epsilon_i^d\big(\Ep[V_{ij_1}V_{ij_2}V_{ij_3}]- \Ep[Z_{ij_1}Z_{ij_2} Z_{ij_3}]  \big)     \right| \leq C_{n,2,d}\bigg\},
	\end{align*}
	where   $\bs{\epsilon}^d=(\epsilon_1^d,\ldots,\epsilon_n^d)^{\rm T}$, $d=0,\ldots,D$.

	Lemma \ref{lem: main} is stated below.
	\begin{lemma}\label{lem: main}
		Assume Conditions (L1) to (L3) hold. Then, for $d = 0,\ldots,D-1$ and a constant $\phi>0$ with
		\begin{equation}\label{eq: phi restriction}
			C_eB_n\phi \ln^2(p n)\leq \sqrt n,
		\end{equation}
		we have
		\begin{align*}
			\rho_{\bs\epsilon^d} &\lesssim k^2\frac{\sqrt{\ln (p(1\vee\phi))}}{\phi} + k\delta + \frac{B_n^2\phi^4\ln^5(p n)}{n^2} +  \left( \Ep[\rho_{\bs\epsilon^{d+1}}\mid \bs{\epsilon}^d] + k^2\frac{\sqrt{\ln (p(1\vee\phi))}}{\phi}+\delta \right)\\
			&\quad \times \left( \frac{C_{n,1,d}\phi^2\ln p}{\sqrt n} + \frac{C_{n,2,d}\phi^3\ln^2p}{n} + \frac{B_n^2\phi^4\ln^3(p n)}{n} \right).
		\end{align*}
	\end{lemma}

	\noindent{\bf Proof of Lemma \ref{lem: main}} 
	We only consider the case of $\phi\geq 1$ since other case is trivial. This, together with \eqref{eq: phi restriction}, implies that
	\begin{equation*}\label{eq: lazy condition}
		C_e B_n \ln^2(p n)\leq \sqrt n.
	\end{equation*}
	
	Fix $d = 0,\ldots,D-1$ and $\bs e^d\in\{0,1\}^n$ such that $\mathcal A_d$ holds if $\bs\epsilon^d = \bs e^d$. All arguments in this proof will be conditional on $\bs\epsilon^d = \bs e^d$. For brevity of notation, we make this conditioning implicit and write $\Pr(\cdot)$ and $\Ep[\cdot]$ instead of $\Pr(\cdot\mid \bs\epsilon^d = \bs e^d)$ and $\Ep[\cdot\mid \bs\epsilon^d = \bs e^d]$, respectively.
	
	
	Recall the definition of $\mathfrak{m}$, see \eqref{def: m}.
	Further, for  $\bs{y}\in\mathbb R^p$, define
	\begin{equation}\label{eq: iy}
		\mathfrak{I} (\bs{y}) =\mathfrak{m}( \bs{S}_{n,\bs \epsilon^d}^V,\bs{y}) -\mathfrak{m}( \bs{S}_n^Z ,\bs{y})
	\end{equation}
	and
	\begin{equation*}\label{eq: h definition}
		\mathfrak{h} (\bs{Y},\bs{y}, \tilde{x}) = \mathds{1}\left\{ -\tilde{x} < \s_k(\bs{Y}-\bs{y})\leq \tilde{x} \right\},\ \text{for  }\tilde{x}\geq 0\text{ and }\bs{Y}\in\mathbb R^p.
	\end{equation*}
	For brevity, if the vector $\bs{y}\in\mathbb R^p$ has the form $\bs{y}=x\bs{1}=(x,\ldots,x)^{\rm T}, x\in\RR$, we simply write $x$ to denote the superscript.
	Also, denote
	\begin{equation*}\label{eq: W definition}
		\bs{W} = \frac{1}{\sqrt n}\sum_{i=1}^n\Big(\epsilon_i^{d + 1}\bs{V}_i + (1-\epsilon_i^{d+1})\bs{Z}_i\Big).
	\end{equation*}
	
	In Step 1 below, we will show that
	\begin{align}
		\sup_{x\in\mathbb R} |\Ep[	\mathfrak{I} (x)]| &\lesssim \frac{B_n^2\phi^4\ln^{5}(p n)}{n^2} +  \left( \Ep[\rho_{\bs\epsilon^{d+1}}] + k\frac{\sqrt{\ln (p(1\vee\phi))}}{\phi}+\delta \right)\nonumber\\
		&\quad \times \left( \frac{C_{n,1,d}\phi^2\ln p}{\sqrt n} + \frac{C_{n,2,d}\phi^3\ln^2p}{n} + \frac{B_n^2\phi^4\ln^3(p n)}{n} \right).\label{eq: EIy bound}
	\end{align}
	In Step 2, we will show that
	\begin{equation}\label{eq: anticoncentration application}
		\rho_{\bs\epsilon^d} \lesssim k^2\frac{\sqrt{\ln (p(1\vee\phi))}}{\phi} +k\delta+ \sup_{x\in\mathbb R}|\Ep[  \mathfrak{I}(x)]|.
	\end{equation}
	The asserted claim follows from combining these two steps. The remaining step is to carry out some auxiliary calculations, which are similar to Steps 3-5 in the
	proof of Lemma 3.1 in \cite{ICL}. Hence we omit the details and only give the procedures of the first two steps here.

	\medskip
	\noindent
	{\bf Step 1.} We adopt the proof strategy from Lemma 3.1 in \cite{ICL} to establish the bound \eqref{eq: EIy bound}. Recall that $I_d = \{i=1,\ldots,n\colon \epsilon_i^d = 1\}$, let $\mathcal S_n$ denote the set of all bijective functions $\{1,\ldots,|I_d|\}$ to $I_d$, and let $\sigma$ be a random function uniformly distributed on $\mathcal S_n$, independent of $\bs{V}_1,\ldots,\bs{V}_n$, $\bs{Z}_1,\ldots,\bs{Z}_n$, and $\bs\epsilon^{d+1}$.
	
	Denote
	$$
	\bs{W}_i^{\sigma} = \frac{1}{\sqrt n}\sum_{j=1}^{i-1}\bs{V}_{\sigma(j)} + \frac{1}{\sqrt n}\sum_{j = i + 1}^{|I_d|} \bs{Z}_{\sigma(j)} + \frac{1}{\sqrt n}\sum_{j\notin I_d}\bs{Z}_j,\ \text{for  }i=1,\ldots,|I_d|.
	$$
	for function $m\colon \mathbb R^p\to\mathbb R$ and   $i\in I_d$, Lemma I.2 in \cite{ICL} implies that
	$$
	\Ep\left[ \frac{\sigma^{-1}(i)}{|I_d|+1}m\Big(\bs{W}^{\sigma}_{\sigma^{-1}(i)} + \frac{\bs{V}_i}{\sqrt n}\Big) + \Big(1 - \frac{\sigma^{-1}(i)}{|I_d|+1}\Big)m\Big(\bs{W}^{\sigma}_{\sigma^{-1}(i)} +\frac{ \bs{Z}_i}{\sqrt n}\Big) \right]
	$$
	is equal to $\Ep[m(\bs{W})]$. Here, Lemma I.2 in \cite{ICL} is applied to the first two terms of $\bs{W}_i^\sigma$ conditional on the set $I_d$ and $\{\bs{Z}_j : j \notin I_d\}$. This property will be used extensively below without further mention.
	
	Fix $x\in\mathbb R$ and observe that
	$$
	\mathfrak{I}(x) = \sum_{i=1}^{|I_d|}\left[ \mathfrak{m} \left(\bs{W}_i^{\sigma} + \frac{\bs{V}_{\sigma(i)}}{\sqrt n},x\right) - \mathfrak{m} \left(\bs{W}_i^{\sigma} + \frac{\bs{Z}_{\sigma(i)}}{\sqrt n},x\right)   \right].
	$$
	Hence, define $f\colon[0,1]\to\mathbb R$ by
	$$
	f(t) = \sum_{i=1}^{|I_d|}\Ep\left[   \mathfrak{m}  \left(\bs{W}_i^{\sigma} + \frac{t \bs{V}_{\sigma(i)}}{\sqrt n},x\right) -   \mathfrak{m} \left(\bs{W}_i^{\sigma} + \frac{t \bs{Z}_{\sigma(i)}}{\sqrt n},x\right) \right],\ \text{for  }t\in[0,1].
	$$
	Then, $\Ep[  \mathfrak{I}(x)] = f(1)$, and by Taylor's expansion,
	$$
	f(1) = f(0) + f^{(1)}(0) + \frac{f^{(2)}(0)}{2} + \frac{f^{(3)}(0)}{6} + \frac{f^{(4)}(\tilde t)}{24},\ \text{where }\tilde t\in(0,1).
	$$
	Here, $f(0) = 0$ by construction and $f^{(1)}(0) = 0$ because $\Ep[V_{i j}] = \Ep[Z_{i j}] = 0$ for  $i\in I_d$ and $j=1,\ldots,p$. Thus, we need to bound $|f^{(2)}(0)|$, $|f^{(3)}(0)|$, and $|f^{(4)}(\tilde t)|$. Using the same reasoning as Steps 3, 4, and 5 in the proof of Lemma 3.1 in \cite{ICL}, we can obtain that
	\begin{align*}
		|f^{(2)}(0)| & \lesssim \frac{B_n^{2}\phi^4\ln^{5}(p n)}{n^{2}} \nonumber\\
		&\quad + \Big(\Ep[\rho_{\bs\epsilon^{d+1}}] + k\frac{\sqrt{\ln (p(1\vee\phi))}}{\phi}+\delta\Big) \Big(\frac{C_{n,1,d}\phi^2\ln p}{\sqrt n} + \frac{B_n^2\phi^4\ln^3(p n)}{n}\Big),
	\end{align*}
	\begin{align*}
		|f^{(3)}(0)| &\lesssim \frac{B_n^{2}\phi^4 \ln^{5}(p n)}{n^{2}} \nonumber\\
		&\quad + \Big(\Ep[\rho_{\bs\epsilon^{d+1}}] + k\frac{\sqrt{\ln (p(1\vee\phi))}}{\phi}+\delta\Big)\Big(\frac{C_{n,2,d}\phi^3\ln^2p}{n} + \frac{B_n^3\phi^5\ln^5(pn)}{n^{3/2}}\Big),
	\end{align*}
	and
	\begin{equation*}
		|f^{(4)}(\tilde t)| \lesssim \frac{B_n^2\phi^4\ln^3 p}{n^2} + \left( \Ep[\rho_{\bs\epsilon^{d+1}}] + k^2\frac{\sqrt{\ln (p(1\vee\phi))}}{\phi}+\delta \right) \frac{B_n^2\phi^4\ln^3 p}{n},
	\end{equation*}
	respectively. Combining these inequalities yields \eqref{eq: EIy bound} and completes Step 1.
	
	\medskip
	\noindent
	{\bf Step 2.} Now we prove \eqref{eq: anticoncentration application}. Fix $x\in\mathbb R$, and observe that
	\begin{align*}
		&\Pr(\s_k(\bs{S}_{n,\bs\epsilon^d}^V) \leq x) \leq\Pr( \mathfrak{q}_k( \bs{S}_{n,\bs\epsilon^d}^V - x - k\phi^{-1}) \leq 0)
		\leq \Ep[\mathfrak{m}( \bs{S}_{n,\bs\epsilon^d}^V,x + k\phi^{-1}) ]\\
		&\quad \leq \Ep[ \mathfrak{m} ( \bs{S}_n^{Z},\bs{x} + k\phi^{-1})] + |\Ep[\mathfrak I(\bs{x}+k\phi^{-1})]|
		\leq \Pr( \s_k(\bs{S}_n^{Z}) \leq x+ 2k\phi^{-1}) + |\Ep[\mathfrak I(\bs{x}+k\phi^{-1})]|\\
		&\quad \leq \Pr( \s_k(\bs{S}_n^{Z} )\leq x) + 2C_ak^2\phi^{-1}\sqrt{\ln p\vee\ln (k^{-1}p\phi/2)}+C_ak\delta + |\Ep[ \mathfrak I(\bs{x}+k\phi^{-1})]|\\
		&\quad \leq \Pr( \s_k(\bs{S}_n^{Z} )\leq x) + 2C_ak^2\phi^{-1}\sqrt{\ln (p(1\vee\phi))}+C_ak\delta+ |\Ep[\mathfrak I(\bs{x}+k\phi^{-1}) ]|.
	\end{align*}
	Here, the first inequality follows from \eqref{eq: f properties}, the second from $\mathfrak{m}(\cdot, \bs{y} + \phi^{-1} ) = g(\mathfrak{q}_k(\cdot - \bs{y} - \phi^{-1}))$ and \eqref{eq: g property}, the third from \eqref{eq: iy}, the fourth from \eqref{eq: g property} and \eqref{eq: f properties}, and the fifth from Condition (L3). The opposite direction follows similarly. Combining these bounds yields \eqref{eq: anticoncentration application} and completes Step 2.
\rulex
	
	\vspace{2ex}
	
	Arguments similar to \cite{ICL} lead to the following Corollary \ref{cor: main} and Lemma \ref{lem: closing}, and hence we omit the relevant proofs.
	
	\begin{corollary}\label{cor: main}
		Assume that all assumptions of Lemma \ref{lem: main} are satisfied. Then there exists a constant $K>0$  such that for  $d = 0,\ldots,D-1$, if $C_{n,1,d+1} \geq C_{n,1,d} + KB_n\ln^{1/2}(p n)$ and $C_{n,2,d+1} \geq C_{n,2,d} + KB_n^2\ln^{3/2}(pn)$, then for   constant $\phi>0$ satisfying \eqref{eq: phi restriction}, it holds that
		\begin{align*}
			\Ep[\rho_{\bs\epsilon^d}\mathds{1}\{\mathcal A_d\}] &\lesssim k^2\frac{\sqrt{\ln p(1\vee \phi)}}{\phi} + k\delta+ \frac{B_n^2\phi^{4}\ln^5(p n)}{n^2} \nonumber\\
			&+ \left( \Ep[\rho_{\bs\epsilon^{d+1}} \mathds{1}\{\mathcal A_{d+1}\}] + k^2\frac{\sqrt{\ln p(1\vee \phi)}}{\phi} +\delta\right)  \nonumber\\
			&\times \left( \frac{C_{n,1,d}\phi^2\ln p}{\sqrt n} + \frac{C_{n,2,d}\phi^3\ln^2p}{n} + \frac{B_n^2\phi^4\ln^3(p n)}{n} \right). \label{eq: max induction}
		\end{align*}
	\end{corollary}

	\begin{lemma}\label{lem: closing}
		For   constant $\phi > 0$ such that \eqref{eq: phi restriction} holds,it follows that $\Ep[\rho_{\bs\epsilon^D}\mathds{1}\{\mathcal A_D\}] \leq 1/n$.
	\end{lemma}
	
	\vspace{2ex}
	
	\noindent{\bf Proof of Proposition \ref{cor: max}} 	
	The proof of Proposition \ref{cor: max} proceeds as follows. In Lemma \ref{lem: main} and Corollary \ref{cor: main}, we derive a recursive inequality for $\Ep[\rho_{\bs\epsilon^d}\mathds{1}\{\mathcal A_d\}]$, $d=0,\ldots,D$. Next, in Lemma \ref{lem: closing}, we prove  that $\Ep[\rho_{\bs\epsilon^D}\mathds{1}\{\mathcal A_D\}]$ is bounded by $1/n$. Then, we use a backward induction argument  to establish a bound for $\Ep[\rho_{\bs\epsilon^0}\mathds{1}\{\mathcal A_0\}]$. Noting that $\epsilon^0_i=1$, this leads to the desired claim once the constants $C_{n,1,d}$ and $C_{n,2,d}$ are appropriately chosen. 

	Assume that
	\begin{equation}\label{eq: natural bound}
		C_e^4B_n^2\ln^5(p n)\leq n
	\end{equation}
	since otherwise the conclusion of the proposition is trivial.
	Suppose that there exits a constant $C_m>0$ such that
	\begin{equation}\label{eq: bn bounds 1}
		\max_{1\leq j_1,j_2\leq p}\left| \frac{1}{\sqrt n}\sum_{i=1}^n\big(\Ep[V_{i j_1}V_{i j_2}] - \Ep[Z_{i j_1}Z_{i j_2}]\big) \right| \leq C_m B_n\sqrt{\ln(p n)}
	\end{equation}
	and
	\begin{equation}\label{eq: bn bounds 2}
		\max_{1\leq j_1,j_2,j_3\leq p}\left| \frac{1}{\sqrt n}\sum_{i=1}^n \big(\Ep[V_{i j_1}V_{i j_2}V_{i j_3}] - \Ep[Z_{i j_1}Z_{i j_2}Z_{i j_3}]\big) \right| \leq C_m B_n^2\sqrt{\ln^3(p n)}.
	\end{equation}
	
	Let $K$ be the constant from Corollary \ref{cor: main} and for  $d = 0,\ldots,D$, define
	\begin{equation*}\label{eq: fancy b bounds}
		C_{n,1,d}= (C_m + K)(d+1)B_n\ln^{1/2}(p n) \quad \text{and} \quad C_{n,2,d}= (C_m + K)(d+1)B_n^2\ln^{3/2}(p n).
	\end{equation*}
	This ensures that $\mathcal A_0$ holds by \eqref{eq: bn bounds 1} and \eqref{eq: bn bounds 2}, and the requirements of Corollary \ref{cor: main} on $C_{n,1,d}$ and $C_{n,2,d}$ are also satisfied.
	
	Now, for  $d = 0,\ldots,D$, define
	$$
	f_d = \inf\left\{x\geq 1\colon \Ep[\rho_{\bs\epsilon^d}\mathds{1}\{\mathcal A_d\}] \leq
	x\left\{k^2\left(\frac{B_n^2\ln^5(p n)}{n}\right)^{\frac{1}{4}}+k\delta\right\}\right\}.
	$$
	Note that $f_d<\infty$ since $\rho_{\bs\epsilon^d}\leq1$. Then, for  $d = 0,\ldots,D-1$, apply Corollary \ref{cor: main} with
	$$
	\phi = \phi_d = \frac{n^{\frac{1}{4}}}{B_n^{1/2}\ln^{3/4}(p n)((d+1)f_{d+1})^{1/3}},
	$$
	which satisfies the required condition \eqref{eq: phi restriction} by the assumption \eqref{eq: natural bound}. Since
	\begin{align*}
		\frac{B_n^2\phi_d^4\ln^{5}(pn)}{n^2}
		& \leq\frac{\ln^{2}(p n)}{n}
		\leq\frac{\ln^{\frac{1}{4}}(p n)}{n^{\frac{1}{4}}}
		\leq C_e((d+1)f_{d+1})^{1/3}\left(\frac{B_n^2\ln^5(p n)}{n}\right)^{\frac{1}{4}}, \\
		\frac{C_e\sqrt{\ln (p\phi_d)}}{\phi_d} &\leq C_e((d+1)f_{d+1})^{1/3}\left(\frac{B_n^2\ln^5(p n)}{n}\right)^{\frac{1}{4}}, \\
		\frac{C_{n,1,d}\phi_d^2\ln p}{\sqrt n} &\leq \frac{(C_m + K)(d+1)}{((d+1)f_{d+1})^{2/3}},
	\end{align*}
	and
	$$	\frac{C_{n,2,d}\phi^3_d\ln^2p}{n}  \bigvee \frac{B_n^2\phi_d^4\ln^3(p n)}{n} \leq \frac{(C_m + K)\vee 1}{f_{d+1}},$$
	by  Corollary \ref{cor: main}, we have
	\begin{align*}
		&\Ep[\rho_{\bs\epsilon^d}\mathds{1}\{\mathcal A_d\}]\\
		&\quad\lesssim\Big(f_{d+1}^{2/3} + (d+1)^{2/3} + 1\Big)\left\{k^2\left(\frac{B_n^2\ln^5(p n)}{n}\right)^{\frac{1}{4}}+k\delta\right\}.
	\end{align*}
	Consequently,
	$$
	f_d \lesssim   f_{d+1}^{2/3} + (d+1)^{2/3} + 1,\quad\text{for  }d = 0,\ldots,D-1.
	$$
	Applying Lemma \ref{lem: closing} we have $f_D = 1$ since $B_n\geq 1$ by assumption. Therefore, using a straightforward induction argument, we conclude that 
	$f_d \lesssim  d$, for $d = 0,\ldots,D$.
	In particular, this implies
	$$
	\rho_{\bs\epsilon^0}\mathds{1}\{\mathcal A_0\} = \Ep[\rho_{\bs\epsilon^0}\mathds{1}\{\mathcal A_0\}] \lesssim
	k^2\left(\frac{B_n^2\ln^5(p n)}{n}\right)^{\frac{1}{4}}+k\delta.
	$$
	Since $\mathcal A_0$ holds by construction, we have $\mathds{1}\{\mathcal A_0\} = 1$. Combining this inequality with the definition of $\rho_{\bs\epsilon^0}$ yields the desired bound. \rulex
	
	\vspace{2ex}

	\subsection{Proof of  Proposition \ref{Proposition: anticoncentration}}
	Throughout	the proof,
	we use ``$J_1 \Rightarrow J_2$'' to denote that $J_1$ is a sufficient condition for $J_2$, and ``$J_1 \Leftrightarrow J_2$'' to indicate that $J_1$ and $J_2$ are logically equivalent.

	In order to obtain the anti-concentration bound for
	$ Y_{(k)} $ in \cite{DAB}, an auxiliary random variable is defined in \cite{DAB}, as provided below.
	For a non-empty subset $A\subseteq \Omega_p:=\{1,\ldots,p\}$, let $\iota(A)$ be a randomly chosen, uniformly distributed, element of $A$. In addition, it is assumed to be independent of $\bs{Y}$ and $\iota(A')$ for other $A'\subseteq \Omega_p$. Let $A^*\in \underset{A\subseteq \Omega_p,|A|=k }{\operatorname{argmax}}\suml_{j\in A}Y_j$ be an element chosen uniformly at random from the argmax set if it is not a singleton.  Abbreviate $\iota(A^*)$ as $\iota^*$.
	Thus $Y_{\iota^*}$ is a randomized relative of $Y_{(k)}$   and $\Pr(Y_{\iota^*}=Y_{(k)} )\ge 1/k$. For each Borel measurable subset $B\subseteq\RR$, it can be obtained that
	\begin{equation}\label{coupling inequality}
		\begin{aligned}
			\Pr(Y_{\iota^*}\in B)
			\ge \Pr(Y_{\iota^*}=Y_{(k)},Y_{(k)}\in B)
			=\Pr(Y_{\iota^*}=Y_{(k)})\Pr(Y_{(k)}\in B)
			\ge \frac{1}{k}\Pr(Y_{(k)}\in B).
		\end{aligned}
	\end{equation}
	
	Such a randomization forms the basis for the decomposition of the corresponding density, that is $\tilde{f}_{k}(w)=\phi(w)\tilde{G}_{k}(w)$, where $\tilde{G}_{k}(w)$ is non-decreasing. This recovers the applicability of the methods in \cite{CAB} so as to provide an anti-concentration bound for $Y_{\iota^*}$ and even $Y_{(k)}$. We state the decomposition formally in the following lemma.
	
	\begin{lemma}[]\label{Lemma:density formula}
		Let $\bs{W}=(W_1,\ldots , W_p)^{\rm T}$ be a (not necessarily centered) Gaussian random vector in $\mathbb{R}^p$ with $\Var(W_j) = 1$ for $1 \le j \le p$. Then the distribution of  $\widetilde{W}_{(k)}:=W_{\iota^*}$ is absolutely continuous with respect to the Lebesgue measure and a version of the density is given by
		\begin{align}\label{eq:density factorization}
			\tilde{f}_{k}(w)=\phi(w)\tilde{G}_{k}(w),
		\end{align}
		where
		\begin{align*}
			\tilde{G}_{k}(w)=\frac{1}{k}\suml_{j=1}^p{\rm e}^{\Ep[W_j]w-\Ep[W_j])^2/2 }\cdot \Pr\left(j \in A^{*} \mid W_{j}=w\right).
		\end{align*}
		Moreover, suppose that $\Ep[W_j]\ge 0$ for $1 \le j \le p$, then the map $w\mapsto \tilde{G}_{k}(w)$ is non-decreasing on $\RR$.
	\end{lemma}
	
	\noindent{\bf Proof of Lemma \ref{Lemma:density formula}. }
	For each Borel measurable subset $B$ of $\RR$, the absolute continuity of the distribution of $\widetilde{W}_{(k)}$ follows from the fact that $\Pr(\widetilde{W}_{(k)} \in B) \le \suml_{j=1}^p\Pr(W_j\in B)$.  Following the proof of Lemma 5 in \cite{CAB} and Lemma 1 in \cite{DAB}, the following expression
	$$\tilde{f}_{k}(w)=\phi(w)\suml_{j=1}^p{\rm e}^{\Ep[W_j]w-(\Ep[W_j])^2/2 }\cdot\Pr(j=\iota^*|W_j=w)$$
	is well-defined.
	
	Moreover, note that $\{w<W_{\iota^*}\le w+\epsilon\}=\bigcup\limits_{j=1}^p\{\iota^*=j\text{ and }w<W_j\le w+\epsilon\}$ and the events in the union are disjoint. Hence we have
	\begin{align*}
		\Pr\left(j=\iota^{*} \mid W_{j}=w\right)&=\Pr\left(j \in A^{*}, j=\iota^{*} \mid W_{j}=w\right)\\
		&=\Pr(j=   \left.\iota^{*} \mid j \in A^{*}, W_{j}=w\right) \Pr\left(j \in A^{*} \mid W_{j}=w\right)\\
		&=\frac{1}{k} \Pr\left(j \in A^{*} \mid W_{j}=w\right),
	\end{align*}
	and
	\begin{align*}
		\Pr(w<W_{\iota^*}\le w+\epsilon)&=\suml_{j=1}^p\Pr(\iota^*=j\text{ and }w<W_j\le w+\epsilon)\\
		&=\suml_{j=1}^p\dint_{w}^{w+\epsilon}\Pr(j=\iota^*|W_j=u)\phi(u-\mu_j)\df u.
	\end{align*}
	Therefore,
	\begin{align*}
		\tilde{f}_{k}(w)&=\suml_{j=1}^p\Pr(j=\iota^*|W_j=w)\phi(w-\mu_j)
		\\
		&=\phi(w)\suml_{j=1}^p{\rm e}^{\mu_j w-\mu_j^2/2}\cdot\Pr(j=\iota^*|W_j=w)\\
		&=\frac{1}{k}\phi(w)\suml_{j=1}^p{\rm e}^{\mu_j w-\mu_j^2/2}\cdot \Pr\left(j \in A^{*} \mid W_{j}=w\right)\\
		&=\phi(w)\tilde{G}_{k}(w).
	\end{align*}
	
	Now we show the second claim. Since $\mu_j=\Ep[W_j]\ge 0$, the map $w\mapsto \exp\{\Ep[W_j]w-(\Ep[W_j])^2/2\} $ is non-decreasing. Thus it suffices to show that the map
	\begin{align}\label{non-decreasing map}
		w\mapsto \Pr\left(j \in A^{*} \mid W_{j}=w\right)
	\end{align}
	is non-decreasing. Note that
	\begin{align*}
		\Pr\left(j \in A^{*} \mid W_{j}=w\right)=\Pr\big(\minl _{l \in A}W_{l} \le W_{j}  \text{ for  } |A|=k,  A \subseteq\{1,\ldots,p\}\,\big| W_{j}=w \big).
	\end{align*}
	Let $\bar{W}_j=W_j-\Ep(W_j)=W_j-\mu_j$, so that $\bar{W}_j, j=1,\ldots,p$ are standard Gaussian random variables. Let $W^o_{j_1 j_2}=\bar{W}_{j_2}-\Ep(\bar{W}_{j_1}\bar{W}_{j_2})\bar{W}_{j_1}$ be the residual from the orthogonal projection of $\bar{W}_{j_2}$ on $\bar{W}_{j_1}$, for $1\le j_1,j_2\le p$. It is  independent of $\bar{W}_j$ or $W_j$ since ${\bs W}$ is jointly Gaussian. Furthermore, for given set  $A\subseteq \{1,\ldots,p\}$ and $1\le j_1 \le p$, based on condition that $W_{j_1}=w$,
	\begin{align*}
		\minl _{j_2 \in A}W_{j_2} \le W_{j_1}&\Leftrightarrow \minl _{j_2 \in A}\{\bar{W}_{j_2}+\mu_{j_2}\} \le w\\
		&\Leftrightarrow\minl _{j_2 \in A}\{W^o_{j_1 j_2} +\Ep[\bar{W}_{j_1}\bar{W}_{j_2}]\bar{W}_{j_1}+\mu_{j_2} \} \le w\\
		&\Leftrightarrow \minl _{j_2 \in A}\{W^o_{j_1 j_2} +\Ep[\bar{W}_{j_1}\bar{W}_{j_2}](w-\mu_{j_1})+\mu_{j_2} \} \le w\\
		&\Leftrightarrow  \minl _{j_2 \in A}\{(\Ep[\bar{W}_{j_1}\bar{W}_{j_2}]-1)w+W^o_{j_1 j_2}-\Ep[\bar{W}_{j_1}\bar{W}_{j_2}]\mu_{j_1}+\mu_{j_2} \} \le 0.
	\end{align*}
	By the independence of  $\left\{W^o_{j_1 j_2}:1\le  j_2\le p\right\}$  from  $W_{j_1} $ and the inequality $\Ep\left[\bar{W}_{j_2} \bar{W}_{j_1}\right]-1\le 0$,  the map (\ref{non-decreasing map})  is   non-decreasing. \rulex
	
	\vspace{2ex}

	To show Proposition \ref{Proposition: anticoncentration}, we also need the Lemma \ref{GCI}; see Theorem 7.1 in \cite{Ledoux2001} for its proof.
	
	\begin{lemma}[Gaussian Concentration Inequality]\label{GCI}
		Let $(Y_1,\ldots, Y_p)^{\rm T}$ be a centered Gaussian random vector in $\RR^p$ with $\maxl_{1 \le j \le p}\Ep[Y_j^2]\le \sigma^2$ for some $\sigma^2>0$. Then for $r > 0$,
		\begin{equation*}
			\Pr\big (\max_{1 \le j \le p} Y_{j}\ge \Ep \big[ \max_{1 \le j \le p} Y_{j}\big] +r\big )\le {\rm e}^{-r^2/(2\sigma^2)}.
		\end{equation*}	
	\end{lemma}
	Observe that when $r\le 0$, $r_+=0$, then the right-hand side equals 1. Hence we have a simple consequence of this inequality: for
	$r\in \RR$ and $1\le k\le p$,
	\begin{equation}\label{eq:GCI-k}
		\begin{aligned}
			\Pr \big( Y_{(k)}\ge \Ep[\|\bs{Y}\|_\infty]+r \big)&\le \Pr \big( Y_{(1)}\ge \Ep[\|\bs{Y}\|_\infty]+r \big)\\
			&\le\Pr\left ( Y_{(1)}\ge \Ep \left [  Y_{(1)}\right ] +r\right )\le {\rm e}^{-r_+^2/(2\sigma^2)}.
		\end{aligned}
	\end{equation}
	
	\vspace{2ex}
	
	We are now in position to prove Proposition \ref{Proposition: anticoncentration}.
	
	\noindent{\bf   Proof of Proposition \ref{Proposition: anticoncentration}}
	Following the proof of Theorem 3 in \cite{CAB}, we also show Proposition \ref{Proposition: anticoncentration} using three steps.
	
	{\bf Step 1:} This step reduces the reasoning to the unit variance case. Choose a $y\ge 0$ and define $W_j:=(Y_j-y)/\sigma_j+y/\uline{\sigma}$. It is clear that $\mu_j:=\Ep[W_j]=y/\underline{\sigma}-y/\sigma_j\ge 0$ and $\Var(W_j)=1$. Let $\{i_1,\ldots,i_p\}$ be a permutation of $\{1,\ldots,p\}$ such that $Y_{i_1}\ge Y_{i_2}\ge\ldots\ge Y_{i_k}\ge\ldots \ge Y_{i_p}$, i.e., $Y_{(k)}=Y_{i_k}$. It is easy to obtain that $\{Y_{(k)}-y|\le \epsilon\}\subseteq \{W_{(k)}-y/\underline{\sigma}|\le \frac{\epsilon}{\underline{\sigma}}\}$. Note that for $k\le j\le p$,
	\begin{align*}
		Y_{i_k}\le y+\epsilon\Rightarrow Y_{i_j}\le y+\epsilon\le y+\epsilon\sigma_{i_j}/\underline{\sigma}  \Leftrightarrow\dfrac{Y_{i_j}-y}{\sigma_{i_j}}\le \dfrac{\epsilon}{\underline{\sigma}}\Rightarrow W_{(k)}-y/\underline{\sigma}\le \dfrac{\epsilon}{\underline{\sigma}};
	\end{align*}
	and for $1\le j\le k$,
	\begin{align*}
		Y_{i_k}\ge y-\epsilon\Rightarrow Y_{i_j}\ge y-\epsilon\ge y-\epsilon\sigma_{i_j}/\underline{\sigma} \Leftrightarrow\dfrac{Y_{i_j}-y}{\sigma_{i_j}}\ge \dfrac{-\epsilon}{\underline{\sigma}}\Rightarrow W_{(k)}-y/\underline{\sigma}\ge \dfrac{-\epsilon}{\underline{\sigma}}.
	\end{align*}
	Combing these two inequalities lead to
	\begin{align}\label{eq:Step 1}
		\Pr(|Y_{(k)}-y|\le \epsilon)\le \Pr\left(|W_{(k)}-y/\underline{\sigma}|\le \frac{\epsilon}{\underline{\sigma}}\right)\le\sup\limits_{x \in \RR}\Pr\left(|W_{(k)}-x|\le \frac{\epsilon}{\underline{\sigma}}\right).
	\end{align}
	
	{\bf Step 2}: This step bounds the density of $\widetilde{W}_{(k)}$, which is the auxiliary random variable defined in \cite{DAB} as a randomized relative of $W_{(k)}$.
	Note that $\Var(W_j)=1$ for $1\le j\le p$. According to Lemma \ref{Lemma:density formula}, $\widetilde{W}_{(k)}$ has density of the form
	\begin{align*}
		\tilde{f}_{k}(w)=\phi(w)\tilde{G}_{k}(w),
	\end{align*}
	where the map $w\mapsto \tilde{G}_{k}(w)$  is non-decreasing.
	Define $\bar{w}:=(1/\underline{\sigma}-1/\overline{\sigma} )|y|\ge 0$, so that $\mu_j\le \bar{w}$ for $1 \le j \le p$. Moreover, define $\bar{W}:=\maxl_{1 \le j \le p}(W_j-\mu_j)$ and $\bar{a}_p:=\Ep\big[ \maxl_{1 \le j \le p} |Y_j/\sigma_j|\big]=\Ep\big[ \maxl_{1 \le j \le p} |W_j-\mu_j|\big]$. Then we have
	\begin{align*}
		\int_{w}^{\infty}\phi(t)\tilde{G}_{k}(w)\dt&\le	\int_{w}^{\infty}\phi(t)\tilde{G}_{k}(t)\dt=\Pr(\widetilde{W}_{(k)}>w)\le \Pr(W_{(1)}>w)\\
		&\le\Pr(\bar{W}\ge w-\bar{w})=\Pr(\bar{W}\ge\bar{a}_p+w-\bar{w}-\bar{a}_p)\\
		&\le\exp\left\{-\frac{(w-\bar{w}-\bar{a}_p)_+^2}{2}\right\},
	\end{align*}
	where the second inequality is due to $\widetilde{W}_{(k)}\le W_{(1)}$, the last inequality is due to (\ref{eq:GCI-k}) by setting $r=w-\bar{w}-\bar{a}_p$. Therefore, for $w\in\RR$,
	\begin{align}\label{eq:bound on Gk}
		\tilde{G}_{k}(w)\le \frac{1}{1-\Phi(w)}\exp\left\{-\frac{(w-\bar{w}-\bar{a}_p)_+^2}{2}\right\}.
	\end{align}
	For $w>0$, the Mill's inequality implies the univariate Mills ratio $$M(w)=\phi(w)\(\int_{w}^{\infty}\phi(t)\dt\)^{-1}=\dfrac{\phi(w)}{1-\Phi(w)}\in\left[w,w\dfrac{1+w^2}{w^2}\right]. $$
	When $w > 1$, $(1+w^2)/w^2\le 2$. On the other hand, $\phi(w)/(1-\Phi(w))\le\phi(1)/(1-\Phi(1))<1.53<2$ for $w<1$. Therefore,
	\begin{align*}
		\phi(w)/(1-\Phi(w))\le 2(w\vee 1), \,  w\in \RR.
	\end{align*}
	Combining the preceding inequality, (\ref{eq:density factorization}), and (\ref{eq:bound on Gk}) yields
	\begin{align*}
		\tilde{f}_{k}(w)\le 2(w\vee 1)\exp\left\{-\frac{(w-\bar{w}-\bar{a}_p)_+^2}{2}\right\},\,  w\in \RR.
	\end{align*}
	{\bf Step 3}: By Step 2, for $x\in\RR$ and $t > 0$, we obtain
	\begin{align*}
		\Pr(|\widetilde{W}_{(k)}-x|\le t)=\int_{x-t}^{x+t}\tilde{f}_{k}(w)\df w\le 2t\max_{|w-x|\le t}\tilde{f}_{k}(w)\le 4t(\bar{w}+\bar{a}_p+1),
	\end{align*}
	where the last inequality due to that the map $w\mapsto w{\rm e}^{-(w-a)^2/2}$ (with $a \ge 0$) is non-increasing on $[a+1,\infty)$. Choose $a=\bar{w}+\bar{a}_p\ge 0$. Then we have the following inequalities:
	\begin{align*}
		&w\ge a+1\Rightarrow w{\rm e}^{-(w-a)^2/2}\le (a+1){\rm e}^{-1/2}<a+1,\\
		&a\le w\le a+1\Rightarrow w{\rm e}^{-(w-a)^2/2}\le w\le a+1,\\
		&w\le a\Rightarrow \tilde{f}_{k}(w)\le 2(w\vee 1)\le 2(a\vee 1)<2(a+1),\\
		&w\ge a\vee1\Rightarrow\tilde{f}_{k}(w)\le 2w{\rm e}^{-(w-a)^2/2}\le 2(a+1),
	\end{align*}
	and
	$$ a\le w\le 1\Rightarrow \tilde{f}_{k}(w)\le 2{\rm e}^{-(w-a)^2/2}\le 2\le 2(a+1).$$
	
	Combining these inequalities yields $\tilde{f}_{k}(w)\le 2(a+1)$. This, along with Step 1 and inequality (\ref{coupling inequality}) leads to
	\begin{equation}\label{eq:Step 3}
		\begin{aligned}
			\Pr(|Y_{(k)}-y|\le \epsilon)&\le\sup\limits_{x \in \RR}\Pr\left(|W_{(k)}-x|\le \frac{\epsilon}{\underline{\sigma}}\right)\\
			&\le k\sup\limits_{x \in \RR}\Pr\left(|\widetilde{W}_{(k)}-x|\le \frac{\epsilon}{\underline{\sigma}}\right)\\
			&\le 4k\epsilon\{(1/\uline{\sigma}-1/\overline{\sigma})|y|+\bar{a}_p+1\}/\uline{\sigma},
		\end{aligned}
	\end{equation}
	for $y\ge 0$ and $\epsilon > 0$.
	
	For the case of $y< 0$, the proof is similar by making a slight modification. Let $W_j:=(Y_j-y)/\sigma_j+y/\overline{\sigma},\mu_j:=\Ep[W_j]=y/\overline{\sigma}-y/\sigma_j,0\le \mu_j\le \bar{w},\Var(W_j)=1$. By the similar argument, (\ref{eq:Step 1}) becomes
	\begin{align*}
		\Pr(|Y_{(k)}-y|\le \epsilon)\le \Pr\left(|W_{(k)}-y/\overline{\sigma}|\le \frac{\epsilon}{\underline{\sigma}}\right)\le\sup\limits_{x\in \RR}\Pr\left(|W_{(k)}-x|\le \frac{\epsilon}{\underline{\sigma}}\right).
	\end{align*}
	It is easy to see that  (\ref{eq:Step 3}) also holds for $y < 0$, so that  (\ref{eq:Step 3}) holds for $y\in\RR$. The above assertions are true by replacing $\bar{a}_p$ with $a_p$. If $\uline{\sigma}=\overline{\sigma}=\sigma$, then for $\epsilon>0$, we have
	\begin{align*}
		\Pr \left( \big|Y_{(k)} - y\big|  \le   \epsilon \right) \le 4k \epsilon (a_p+1)/\sigma,\,  y\in\RR,
	\end{align*}
	which implies the first assertion of Proposition \ref{Proposition: anticoncentration}.
	
	We next consider the case $\uline{\sigma}<\overline{\sigma}$.  Suppose  that $0<\epsilon\le\uline{\sigma}$. By the symmetry of $\{Y_j:1\le j\le p\}$ and the assumption $k\le \lfloor\frac{p+1}{2}\rfloor, Y_{[p-k+1]}\le Y_{(k)}$, we have $\Pr(Y_{(k)}\ge a)=\Pr(Y_{[p-k+1]}\le -a)\ge \Pr(Y_{(k)}\le -a)$ for $a\in \RR$. Further by (\ref{eq:GCI-k}), we have $\Pr(|Y_{(k)}-y|\le \epsilon)\le \Pr(Y_{(k)}\ge y-\epsilon)=\Pr(Y_{(k)}\ge |y|-\epsilon)$ for $y\ge \epsilon+\overline{\sigma}(\bar{a}_p+\sqrt{2\ln(\uline{\sigma}/\epsilon)})$. Similarly, if $y\le -\epsilon-\overline{\sigma}(\bar{a}_p+\sqrt{2\ln(\uline{\sigma}/\epsilon)})$, then $\Pr(|Y_{(k)}-y|\le \epsilon)\le \Pr(Y_{(k)}\le y+\epsilon)\le \Pr(Y_{(k)}\ge -y-\epsilon)=\Pr(Y_{(k)}\ge |y|-\epsilon)$. Moreover $\bar{a}_p=\Ep\big[\maxl_{1\le j\le p}|Y_j/\sigma_j|\big]\ge \Ep\big[\maxl_{1 \le j \le p}|Y_j|/\overline{\sigma}]$. Combining these results  yields
	\begin{equation}\label{eq: |y| larger}
		\begin{aligned}
			\Pr(|Y_{(k)}-y|\le \epsilon)&\le\Pr(Y_{(k)}\ge |y|-\epsilon)\le \Pr(Y_{(k)}\ge \overline{\sigma}(\bar{a}_p+\sqrt{2\ln(\uline{\sigma}/\epsilon)}))\\
			&\le\Pr(Y_{(k)}\ge \Ep\big[\maxl_{1 \le j \le p}|Y_j|\big]+\overline{\sigma}\sqrt{2\ln(\uline{\sigma}/\epsilon)})\le \epsilon/\uline{\sigma}.
		\end{aligned}
	\end{equation}
	For $|y|\le \epsilon+\overline{\sigma}(\bar{a}_p+\sqrt{2\ln(\uline{\sigma}/\epsilon)})$, by (\ref{eq:Step 3}) and the condition $\epsilon\le\uline{\sigma}$, it can be obtained that
	\begin{equation}\label{eq: |y|smaller}
		\begin{aligned}
			\Pr(|Y_{(k)}-y|\le \epsilon)&\le 4k\epsilon\{(\overline{\sigma}/\uline{\sigma})\bar{a}_p+(\overline{\sigma}/\uline{\sigma}-1)\sqrt{2\ln(\uline{\sigma}/\epsilon)}+2-\uline{\sigma}/\overline{\sigma}\}/\uline{\sigma}\\
			&\le 4k\epsilon\cdot\sqrt{2}(\overline{\sigma}/\uline{\sigma}^2)\{\bar{a}_p+\sqrt{\ln(\uline{\sigma}/\epsilon)}+1\}\\
			&\le 8\sqrt{2}(\overline{\sigma}/\uline{\sigma}^2)k\epsilon\{\bar{a}_p+\sqrt{1\vee \ln(\uline{\sigma}/\epsilon)}\}.
		\end{aligned}
		\end{equation}
		Combining (\ref{eq: |y| larger}) and (\ref{eq: |y|smaller}), we derive the inequality in (ii).
		If $\epsilon>\underline\sigma$, the right-hand side in (ii) is larger than one, and hence the same inequality follows trivially.
	
	\vspace{2ex}

	\subsection{Proof of Corollary \ref{cor: anticoncentration}}
	
	Set $\bs Y=\bs Z+\bs\mu$ and $N_t=\sum_{j=1}^p\mathds 1\{Y_j>t\}$.
	For $k=1$, the asserted inequality follows directly from Lemma J.3 of \cite{ICL}, after scaling by $\underline\sigma$, because
	$$
	\{\mathfrak s_1(\bs Z+\bs\mu)\leq x\}=\{\bs Z\leq x\bs 1-\bs\mu\}.
	$$
	Suppose that $k\geq2$, let $\xi_1,\ldots,\xi_p$ be independent Bernoulli$(\pi)$ random variables independent of $\bs Y$, and set $\pi=1/k$.
	Define
	$$
	M_\pi=\max_{j:\xi_j=1}Y_j,
	$$
	where the maximum over the empty set is $-\infty$, and let $H_\pi(t)=\Pr(M_\pi\leq t)$.
	Conditioning on $\bs Y$ and then taking expectation, we have
	\begin{equation}\label{eq: thinning cdf}
	H_\pi(t)=\Ep[(1-\pi)^{N_t}].
	\end{equation}
	
	Observe that
	$$
	x<\mathfrak s_k(\bs Y)\leq x+\epsilon
	\quad\Longleftrightarrow\quad
	N_x\geq k\ \text{and}\ N_{x+\epsilon}\leq k-1.
	$$
	Since $N_{x+\epsilon}\leq N_x$ for every realization of $\bs Y$ and $r\mapsto(1-\pi)^r$ is non-increasing, the random variable
	$$
	D_x=(1-\pi)^{N_{x+\epsilon}}-(1-\pi)^{N_x}
	$$
	is nonnegative.
	
	Moreover, whenever these equivalent conditions hold, $N_{x+\epsilon}\leq k-1$ and $N_x\geq k$, and hence
	$$
	D_x
	\geq (1-\pi)^{k-1}-(1-\pi)^k
	=\pi(1-\pi)^{k-1}.
	$$
	Consequently, \eqref{eq: thinning cdf} gives
	\begin{align*}
	H_\pi(x+\epsilon)-H_\pi(x)
	&=\Ep[D_x]\\
	&\geq\Ep\left[D_x\mathds 1\{x<\mathfrak s_k(\bs Y)\leq x+\epsilon\}\right]\\
	&\geq\pi(1-\pi)^{k-1}\Pr\{x<\mathfrak s_k(\bs Y)\leq x+\epsilon\}.
	\end{align*}
	For $\pi=1/k$, the elementary inequality $(1-1/k)^{k-1}\geq \mathrm e^{-1}$ implies that
	$$
	\pi(1-\pi)^{k-1}
	=\frac{1}{k}\left(1-\frac{1}{k}\right)^{k-1}
	\geq\frac{1}{\mathrm e k}.
	$$
	Rearranging the preceding inequality yields
	\begin{equation}\label{eq: thinning reduction}
	\Pr\{x<\mathfrak s_k(\bs Y)\leq x+\epsilon\}
	\leq \mathrm e k\{H_\pi(x+\epsilon)-H_\pi(x)\}.
	\end{equation}
	Conditional on $\bs\xi=(\xi_1,\ldots,\xi_p)^{\rm T}$, write $S_{\bs\xi}=\{j:\xi_j=1\}$.
	For every nonempty $S_{\bs\xi}$, the selected Gaussian subvector retains the variance lower bound
	$\min_{j\in S_{\bs\xi}}\mathbb E[Z_j^2]\geq\underline\sigma^2$.
	If $|S_{\bs\xi}|\geq2$, Lemma J.3 of \cite{ICL}, applied to $\bs Z_{S_{\bs\xi}}/\underline\sigma$, gives
	\begin{align*}
	&\Pr\left\{\max_{j\in S_{\bs\xi}}(Z_j+\mu_j)\leq x+\epsilon\right\}
	-\Pr\left\{\max_{j\in S_{\bs\xi}}(Z_j+\mu_j)\leq x\right\}\\
	&\qquad\lesssim \frac{\epsilon}{\underline\sigma}\sqrt{1\vee\ln |S_{\bs\xi}|}
	\leq \frac{\epsilon}{\underline\sigma}\sqrt{1\vee\ln p}.
	\end{align*}
	If $S_{\bs\xi}=\{j\}$, then $M_\pi=Z_j+\mu_j$ conditionally on $\bs\xi$.
	Writing $\sigma_j^2=\Ep[Z_j^2]$, the density $f_j$ of $Z_j+\mu_j$ satisfies
	$$
	\sup_{u\in\RR}f_j(u)=\frac{1}{\sqrt{2\pi}\sigma_j}
	\leq\frac{1}{\sqrt{2\pi}\underline\sigma}.
	$$
	Therefore,
	\begin{align*}
	&\Pr(M_\pi\leq x+\epsilon\mid\bs\xi)-\Pr(M_\pi\leq x\mid\bs\xi)\\
	&\qquad=\int_x^{x+\epsilon}f_j(u)\df u
	\leq\frac{\epsilon}{\sqrt{2\pi}\underline\sigma}
	\lesssim\frac{\epsilon}{\underline\sigma}\sqrt{1\vee\ln p}.
	\end{align*}
	If $S_{\bs\xi}$ is empty, then $M_\pi=-\infty$ by convention, so both conditional distribution functions equal one and their increment is zero.
	Thus, for every realization of $\bs\xi$,
	\begin{align*}
	&\Pr(M_\pi\leq x+\epsilon\mid\bs\xi)-\Pr(M_\pi\leq x\mid\bs\xi)\\
	&\qquad\lesssim\frac{\epsilon}{\underline\sigma}\sqrt{1\vee\ln p}.
	\end{align*}
	Taking expectation with respect to $\bs\xi$ and using the law of total probability, we obtain
	\begin{align*}
	H_\pi(x+\epsilon)-H_\pi(x)
	&=\Ep_{\bs\xi}\left[\Pr(M_\pi\leq x+\epsilon\mid\bs\xi)-\Pr(M_\pi\leq x\mid\bs\xi)\right]\\
	&\lesssim\frac{\epsilon}{\underline\sigma}\sqrt{1\vee\ln p}.
	\end{align*}
	Combining this inequality with \eqref{eq: thinning reduction} gives
	$$
	\sup_{x\in\RR}\Pr\{x<\mathfrak s_k(\bs Z+\bs\mu)\leq x+\epsilon\}
	\lesssim\frac{k\epsilon}{\underline\sigma}\sqrt{1\vee\ln p},
	$$
	which proves the one-sided bound.
	Finally, for every $a\in\RR$,
	$$
	\{\mathfrak s_k(\bs Z+\bs\mu)=a\}
	\subseteq\bigcup_{j=1}^p\{Z_j+\mu_j=a\}.
	$$
	Each coordinate $Z_j+\mu_j$ is a Gaussian random variable with positive variance, and hence the probability of the event on the right-hand side is zero.
	Thus $\mathfrak s_k(\bs Z+\bs\mu)$ has no atoms.
	For every $x\in\RR$, it follows that
	\begin{align*}
	\Pr\{|\mathfrak s_k(\bs Z+\bs\mu)-x|\leq\epsilon\}
	&=\Pr\{x-\epsilon<\mathfrak s_k(\bs Z+\bs\mu)\leq x+\epsilon\}\\
	&\lesssim\frac{2k\epsilon}{\underline\sigma}\sqrt{1\vee\ln p}.
	\end{align*}
	Taking the supremum over $x$ and absorbing the factor two into the implicit constant proves the asserted bound for the L\'{e}vy concentration function.
	\rulex
	
	\vspace{2ex}

	\subsection{Proof of Proposition \ref{thm: stein kernel} }

	We adopt the same notations as in the proof of Lemma \ref{lem: main}, such as the constant $\phi>0$, the functions $\mathfrak{m}$ and $\mathfrak{h} $, and the partial derivatives $\mathfrak{m}_{j_1 j_2} $. Here and in what follows, without loss of generality, we will assume that $\bs{V}$ and $\bs{Z}$ are independent.

	Denote
	$$
	\mathcal I(\bs y)=\mathfrak m(\bs V,\bs y)-\mathfrak m(\bs Z,\bs y),\qquad \bs y\in\RR^p.
	$$
	The smoothing argument used in Step 2 of the proof of Lemma \ref{lem: main}, together with Corollary \ref{cor: anticoncentration}, gives, for every $\phi\geq1$,
	\begin{equation}\label{eq: stein new smoothing}
	\sup_{x\in\RR}\left|\Pr(V_{(k)}\leq x)-\Pr(Z_{(k)}\leq x)\right|
	\lesssim k^2\phi^{-1}\sqrt{1\vee\ln p}+\sup_{\bs y\in\RR^p}|\Ep[\mathcal I(\bs y)]|.
	\end{equation}
	Fix $\bs y\in\RR^p$, let $\bs W_t=\sqrt t\bs V+\sqrt{1-t}\bs Z$, and define
	$$
	\Psi(t)=\Ep[\mathfrak m(\bs W_t,\bs y)],\qquad t\in[0,1].
	$$
	The Stein identity and Gaussian integration by parts give
	$$
	|\Psi'(t)|\lesssim k^{-1}\phi^2(\ln p)
	\Ep\left[\mathfrak h\left(\bs y(\bs V,t),\bs Z,\frac{k}{\phi\sqrt{1-t}}\right)
	\max_{1\leq j_1,j_2\leq p}|\tau_{j_1j_2}(\bs V)-\Delta_{j_1j_2}|\right],
	$$
	where $\bs y(\bs V,t)=(1-t)^{-1/2}(\bs y-\sqrt t\bs V)$.
	Conditional on $\bs V$, Corollary \ref{cor: anticoncentration} applies uniformly to the mean shift $-\bs y(\bs V,t)$ and yields
	$$
	\Ep\left[\mathfrak h\left(\bs y(\bs V,t),\bs Z,\frac{k}{\phi\sqrt{1-t}}\right)\bigg|\bs V\right]
	\lesssim \frac{k^2}{\phi\sqrt{1-t}}\sqrt{1\vee\ln p}.
	$$
	Consequently,
	$$
	|\Psi'(t)|\lesssim \frac{k\phi m_\delta\ln^{3/2}p}{\sqrt{1-t}}.
	$$
	Since $\int_0^1(1-t)^{-1/2}\dt=2$, we obtain
	\begin{equation}\label{eq: stein new interpolation}
	\sup_{\bs y\in\RR^p}|\Ep[\mathcal I(\bs y)]|
	\lesssim k\phi m_\delta\ln^{3/2}p.
	\end{equation}
	Combining \eqref{eq: stein new smoothing} and \eqref{eq: stein new interpolation} gives
	$$
	\sup_{x\in\RR}\left|\Pr(V_{(k)}\leq x)-\Pr(Z_{(k)}\leq x)\right|
	\lesssim k^2\phi^{-1}\sqrt{\ln p}+k\phi m_\delta\ln^{3/2}p.
	$$
	If $m_\delta=0$, letting $\phi\to\infty$ proves the claim.
	Otherwise, take $\phi^2=k/(m_\delta\ln p)$.
	When this choice is smaller than one, the asserted bound is trivial because its right-hand side is larger than one.
	Substitution in the remaining case proves the proposition.
	\rulex
	
	\vspace{2ex}

	\subsection{Proof of Proposition \ref{coro:g-g-comparison} }
	
	If $\bs V=\bs Z_1$ and $\bs Z=\bs Z_2$, then the multivariate Stein identity gives $\tau\equiv\Delta_1$ and hence $m_\delta=M_\delta$. Proposition \ref{thm: stein kernel} therefore yields the conclusion directly. \rulex
	
	\vspace{2ex}
	
	\subsection{Proof of Corollary \ref{coro:beta-comparison}}
	
	Let $\tau^\varepsilon$ be the univariate Stein kernel of a standardized Beta$(\alpha,\beta)$ random variable.
	The explicit Beta Stein kernel is bounded, and hence  $\sup_{x\in\operatorname{supp}(\varepsilon_i)}|\tau^\varepsilon(x)|\leq C_{\alpha,\beta}$ and $\Ep[\tau^\varepsilon(\varepsilon_i)]=1$, where $C_{\alpha,\beta}>0$ is a constant depending only on $\alpha$ and $\beta$.
	Below, $c_{\alpha,\beta}>0$ denotes another constant depending only on $\alpha$ and $\beta$.
	By Lemma 4.6 in \cite{K19b}, $\bs V$ admits a Stein kernel satisfying
	$$
	\tau^V_{j_1j_2}(\bs V)=\Ep\left[\frac{1}{n}\sum_{i=1}^n\tau^\varepsilon(\varepsilon_i)a_{ij_1}a_{ij_2}\,\bigg|\,\bs V\right].
	$$
	Define
	$$
	S_{j_1j_2}=\frac{1}{n}\sum_{i=1}^n\{\tau^\varepsilon(\varepsilon_i)-1\}a_{ij_1}a_{ij_2},
	\qquad
	\widetilde M=\max_{1\leq j_1,j_2\leq p}|S_{j_1j_2}|.
	$$
	Since $\Ep[\tau^\varepsilon(\varepsilon_i)]=1$ and
	$\Delta_{j_1j_2}=n^{-1}\sum_{i=1}^na_{ij_1}a_{ij_2}$, we have
	$$
	\tau^V_{j_1j_2}(\bs V)-\Delta_{j_1j_2}
	=\Ep[S_{j_1j_2}\mid\bs V].
	$$
	
	Since the maximum norm is convex, conditional Jensen's inequality gives
	\begin{align*}
	m_\delta
	&=\Ep\left[\max_{1\leq j_1,j_2\leq p}
	\left|\Ep[S_{j_1j_2}\mid\bs V]\right|\right]\\
	&\leq\Ep\left[\Ep\left[
	\max_{1\leq j_1,j_2\leq p}|S_{j_1j_2}|
	\,\bigg|\,\bs V\right]\right]
	=\Ep[\widetilde M].
	\end{align*}
	For each pair $(j_1,j_2)$, the sum inside $\widetilde M$ consists of independent centered terms and
	$$
	\frac{1}{n^2}\sum_{i=1}^na_{ij_1}^2a_{ij_2}^2\leq\frac{1}{n^2}\left(\sum_{i=1}^na_{ij_1}^4\right)^{1/2}\left(\sum_{i=1}^na_{ij_2}^4\right)^{1/2}\leq\frac{C_B^2}{n}.
	$$
	Hence, Hoeffding's inequality and a union bound over the $p^2$ pairs imply that
	$$
	\Pr(\widetilde M>s)\leq2p^2\exp\left(-\frac{c_{\alpha,\beta}ns^2}{C_B^2}\right).
	$$
	\par
	Put $a=c_{\alpha,\beta}n/C_B^2$ and
	$s_0=\{\ln(2p^2)/a\}^{1/2}$, so that
	$2p^2\exp(-as_0^2)=1$.
	Since $\widetilde M$ is nonnegative, the tail-integration formula and the trivial bound
	$\Pr(\widetilde M>s)\leq1$ give
	\begin{align*}
	\Ep[\widetilde M]
	&=\int_0^\infty\Pr(\widetilde M>s)\,\mathrm{d}s\\
	&\leq s_0+2p^2\int_{s_0}^\infty\exp(-as^2)\,\mathrm{d}s\\
	&\leq s_0+\frac{2p^2\exp(-as_0^2)}{2as_0}
	=s_0+\frac{1}{2as_0}
	\lesssim C_B\sqrt{\frac{\ln p}{n}},
	\end{align*}
	where the third line uses
	$\int_t^\infty\exp(-as^2)\,\mathrm{d}s\leq\exp(-at^2)/(2at)$, and the last inequality absorbs constants depending only on $\alpha$ and $\beta$.
	Together with the conditional Jensen bound above, this yields
	\begin{equation}\label{eq: beta kernel first moment}
		m_\delta\lesssim C_B\sqrt{\frac{\ln p}{n}}.
	\end{equation}
	The covariance matrix of $\bs Z$ satisfies $\Delta_{j_1j_2}=n^{-1}\sum_i a_{ij_1}a_{ij_2}$ and $\Delta_{jj}\geq C_b$.
	Combining \eqref{eq: beta kernel first moment} with Proposition \ref{thm: stein kernel} gives
	$$
	\sup_{x\in\mathbb R}\Big|\Pr(V_{(k)}\leq x)-\Pr(Z_{(k)}\leq x)\Big|
	\lesssim k^{3/2}\left(\frac{C_B^2\ln^5 p}{n}\right)^{1/4}
	\leq k^{3/2}\left(\frac{C_B^2\ln^5(pn)}{n}\right)^{1/4}.
	$$
	This proves \eqref{eq: beta distribution comparison}.
	\rulex
	
	\vspace{2ex}

	\vspace{2ex}
	\subsection{Proofs  of Theorem  \ref{thm: gaussian approximation main} and Theorem  \ref{cor: rejection probabilities} }
	%
	%
	The following Lemma  \ref{lem: subgaussian growth} and Lemma  \ref{lem: all conditions}, adapted from \cite{ICL}, serve as essential auxiliary results for proving Theorem \ref{thm: gaussian approximation main} and Theorem  \ref{cor: rejection probabilities} .
	\begin{lemma}\label{lem: subgaussian growth}
		Suppose that Condition (KS1) is satisfied. Then
		\begin{equation}\label{eq: subgaussian bound on x}
			\max_{1\leq i\leq n}\|\bs{X}_i\|_{\infty} \leq 5B_n\ln(p n)
		\end{equation}
		with probability at least $1 - 1/(2n^4)$. In addition,
		$$
		\max_{1\leq i\leq n}\Ep\left[\|\bs{X}_i\|_{\infty}^8\right] \lesssim B_n^8\ln^8(p n).
		$$
	\end{lemma}
	
	\begin{lemma}\label{lem: all conditions}
		Suppose that Conditions  (KS1) and (KS2) are satisfied and set $\tilde{ \bs{X}}_i=(\tilde X_{i 1},\ldots,\tilde X_{i p})^{\rm T} = \bs{X}_i - \overline{\boldsymbol X}$ for  $i = 1,\ldots,n$.
		Then there exist a universal constant $C_u\in(0,1]$ and constants $C>0$ and $n_0\in\mathbb N$ depending only on $b_1$ and $b_2$ such that for  $n\geq n_0$, if the inequality
		\begin{equation}\label{eq: bn restriction once again}
			B_n^2\ln^5(p n)\leq C_u n
		\end{equation}
		holds, then the following inequalities hold jointly with probability at least $1 - 1/n-3\sqrt{\frac{B_n^2\ln^3(pn)}{n}}$,
		\begin{align}
			&\frac{b_1^2}{2}\leq \frac{1}{n}\sum_{i=1}^n \tilde X_{i j}^2 \quad \text{and}  \quad \frac{1}{n}\sum_{i=1}^n \tilde X_{i j}^4 \leq 2B_n^2 b_2^{2}, \quad \text{for  }j=1,\ldots,p
			\label{eq: upper and lower bounds for second moment}; \\
			&\max_{1\leq j_1,j_2\leq p}\left|\frac{1}{\sqrt n}\sum_{i=1}^n(\tilde X_{i j_1}\tilde X_{i j_2} - \Ep[X_{ij_1}X_{i j_2}])\right| \leq CB_n\sqrt{\ln(p n)} \label{eq: second order deviation};  \\
			&\max_{1\leq j_1,j_2,j_3\leq p}\left|\frac{1}{\sqrt n}\sum_{i=1}^n(\tilde X_{i j_1}\tilde X_{i j_2}\tilde X_{i j_3} - \Ep[X_{ij_1}X_{i j_2}X_{i j_3}])\right| \leq CB_n^2\sqrt{\ln^3(p n)}. \label{eq: third order deviation}
		\end{align}
		
	\end{lemma}

	Leveraging the proof strategies of Lemma 4.3 in \cite{ICL},  we obtain the following result, which directly yields   Theorem \ref{thm: gaussian approximation main} and Theorem  \ref{cor: rejection probabilities}.
	
	\begin{lemma}\label{coro: gaussian approximation}
		Assume that Conditions  (KS1)  and  (KS2)  are satisfied. Then
		with probability at least $1-1/(2n^4)-1/n-3\sqrt{B_n^2\ln^3(pn)/n}$,
		\begin{equation}\label{eq: multiplier approximation ks metric1}
			\sup_{x\in\mathbb{R}}|\Pr(T_{(k)}\leq x)-\Pr( G_{(k)}\leq x \mid \bs{X}_1,\ldots,\bs{X}_n)|\lesssim  k^2\left(\frac{B_n^2\ln^5(pn)}{n}\right)^{\frac{1}{4}}
		\end{equation}
		
		In addition, the deterministic bound
		\begin{equation}\label{eq: gaussian approximation ks metric2}
			\sup_{x\in\mathbb{R}}|\Pr(T_{(k)}\leq x)-\Pr( Q_{(k)}\leq x)|\lesssim  k^2\left(\frac{B_n^2\ln^5(pn)}{n}\right)^{\frac{1}{4}}.
		\end{equation}
		holds.
	\end{lemma}
	
	\noindent{\bf Proof of Lemma \ref{coro: gaussian approximation}}
	Without loss of generality, we may assume that \eqref{eq: bn restriction once again} holds and that $n$ is sufficiently large so that $n\geq n_0$ for $n_0$ from Lemma \ref{lem: all conditions}, since otherwise the conclusion of the lemma is trivial by taking $C$ large enough. This will ensure the applicability of Lemma \ref{lem: all conditions} when needed. Moreover, by again taking $C$ large enough, we may assume that $1/n^4 + 2/n + 3\sqrt{\frac{B_n^2\ln^3(pn)}{n}} < 1$.
	
	Let $\mathcal{A}_n$ be the event that \eqref{eq: subgaussian bound on x} and \eqref{eq: upper and lower bounds for second moment}--\eqref{eq: third order deviation} hold simultaneously. 
	
	By Lemma 4.1 and 4.2 in \cite{ICL},
	\begin{equation}\label{eq: probability of An}
		\Pr(\mathcal A_n)\geq 1-\frac{1}{2n^4}-\frac{1}{n}-3\sqrt{\frac{B_n^2\ln^3(pn)}{n}}>0.
	\end{equation}
	
	Further, let $\varepsilon_1,\ldots,\varepsilon_n$ be independent standardized Beta$(1/2,3/2)$ random variables, standardized in such a way that they have mean zero and unit variance (cf. Corollary \ref{coro:beta-comparison}) and are independent of $\bs{X}_{1:n} = (\bs{X}_1,\ldots,\bs{X}_n)$. It is also straightforward to verify that  $\Ep[\varepsilon_i^3]=1$ for  $i=1,\ldots,n$.
	
	Define $T_{(k)}^*=\s_k\big( \frac{1}{\sqrt{n}} \sum_{i=1}^n \varepsilon_i(\bs{X}_i - \overline{\boldsymbol X}) \big)$ as the multiplier bootstrap statistic with weights $\varepsilon_1,\ldots,\varepsilon_n$.
	Condition on $\bs{X}_{1:n}$ such that $\mathcal A_n$ holds.
	Then, combining Corollary \ref{coro:beta-comparison} with the definition of $\mathcal A_n$ gives
	\begin{equation}\label{beta-to-gauss}
		\sup_{x\in\mathbb{R}}\left|\Pr\left(T_{(k)}^*  \leq x \mid \bs{X}_{1:n} \right)-\Pr(G_{(k)}\leq x\mid \bs{X}_{1:n})\right|
		\lesssim 
		k^{3/2}\left(\frac{B_n^2\ln^5(pn)}{n}\right)^{\frac{1}{4}}.
	\end{equation}
	On the other hand, by Proposition \ref{coro:g-g-comparison}, we see that
	\begin{equation*}
		\sup_{x\in\mathbb{R}}|\Pr(G_{(k)}\leq x\mid \bs{X}_{1:n})-\Pr(Q_{(k)}\leq x)|
		\lesssim 
		k^{3/2}\left(\frac{B_n^2\ln^5(pn)}{n}\right)^{\frac{1}{4}}.
	\end{equation*}
	
	Next, we apply Proposition \ref{cor: max} to  compare the distribution of $T_{(k)}$ with the conditional distribution of $T_{(k)}^*$. Formally, consider independent copies $\bs{Y}_1,\ldots,\bs{Y}_n$ of $\bs{X}_1,\ldots,\bs{X}_n$ that are also independent of $\bs{X}_{1:n}$, and define $T_{(k)}'$ by $T_{(k)}$ with $\bs{X}_i$'s replaced by $\bs{Y}_i$'s. By construction, $\Pr (T_{(k)} \le x) = \Pr (T_{(k)}' \le x \mid \bs{X}_{1:n})$. Condition on $\bs{X}_{1:n}$ such that $\mathcal A_n$ holds and apply Proposition \ref{cor: max} with $\bs{V}_i = \bs{Y}_i$ and $\bs{Z}_i = \varepsilon_i \tilde{ \bs{X}}_i$ for  $i=1,\ldots,n$.
	Note that $\Ep[\varepsilon_i]=0$ and $\Ep[\varepsilon_i^2] = \Ep[\varepsilon_i^3]=1$ for  $i=1,\ldots,n$, it follows from the definition of $\mathcal A_n$ that Conditions (L1)  and (L2)  as well as inequalities (\ref{eq: bn bounds 1}) and (\ref{eq: bn bounds 2}) of Proposition \ref{cor: max} are satisfied.
	It remains to verify Condition (L3) in Proposition \ref{cor: max}.
	Observe that for    $x\in\mathbb R$ and $t>0$, there exit   constants $C_1$ and $C_2$ such that
	\begin{equation*}
		\label{beta-anti}
		\begin{split}
			&\quad  \quad\Pr\left(T_{(k)}^*  \leq x+ t \mid \bs{X}_{1:n}\right)\\
			&\quad   \leq \Pr\left( G_{(k)} \leq x+t \mid \bs{X}_{1:n}\right) + C_1k^{3/2}\left(\frac{B_n^2\ln^5(pn)}{n}\right)^{\frac{1}{4}}  \\
			&\quad  \leq \Pr\left( G_{(k)} \leq x\mid \bs{X}_{1:n}\right) + C_2k t\sqrt{1\vee\ln p} + C_1k^{3/2}\left(\frac{B_n^2\ln^5(pn)}{n}\right)^{\frac{1}{4}}\\
			&\quad  \leq\Pr\left(T_{(k)}^* \leq x\mid \bs{X}_{1:n}\right)
			+C_2k t\sqrt{1\vee\ln p} + 2C_1k^{3/2}\left(\frac{B_n^2\ln^5(pn)}{n}\right)^{\frac{1}{4}},
		\end{split}
	\end{equation*}
	where the first and the third inequalities follow from \eqref{beta-to-gauss}, the second from Corollary \ref{cor: anticoncentration} and \eqref{eq: upper and lower bounds for second moment}. 
	
	Since $\sqrt{1\vee\ln p}\lesssim\sqrt{\ln p\vee\ln(p/t)}$, Condition (L3) holds with $\delta\asymp k^{1/2}\{B_n^2\ln^5(pn)/n\}^{1/4}$.
	Consequently,  by applying Proposition \ref{cor: max}, we conclude that
	\begin{align*}
		\sup_{x\in\mathbb{R}}|\Pr(T_{(k)}\leq x)-\Pr(T^*_{(k)}\leq x\mid \bs{X}_{1:n})| &= \sup_{x\in\mathbb{R}}|\Pr(T_{(k)}'\leq x \mid \bs{X}_{1:n})-\Pr(T^*_{(k)}\leq x\mid \bs{X}_{1:n})|\\
		&\lesssim k^2\left(\frac{B_n^2\ln^5(pn)}{n}\right)^{\frac{1}{4}}
	\end{align*}
	The desired result follows from these bounds with the triangle inequality, namely,
	\begin{equation*}
		\begin{array}{ccl}
			\left|\Pr(T_{(k)}\leq x)-\Pr\left(  G_{(k)} \leq x\mid \bs{X}_{1:n}\right)\right|
			&\le&|\Pr(T_{(k)}\leq x)-\Pr(T^*_{(k)}\leq x\mid \bs{X}_{1:n})|\\
			&&+\left| \Pr\left( T^*_{(k)}\leq x\mid \bs{X}_{1:n}\right)-\Pr\left(  G_{(k)} \leq x\mid \bs{X}_{1:n}\right)\right| \\
		\end{array}
	\end{equation*}
	and
	\begin{equation*}
		\begin{array}{ccl}
			|\Pr(T_{(k)}\leq x)-\Pr(Q_{(k)}\leq x)|
			&\le&|\Pr(T_{(k)}\leq x)-\Pr(T^*_{(k)}\leq x\mid \bs{X}_{1:n})|\\
			&&+\left| \Pr\left( T^*_{(k)}\leq x\mid \bs{X}_{1:n}\right)-\Pr\left(  G_{(k)} \leq x\mid \bs{X}_{1:n}\right)\right| \\
			&&+\left|\Pr\left(  G_{(k)} \leq x\mid \bs{X}_{1:n}\right)- \Pr(Q_{(k)}\leq x)\right|.
		\end{array}
	\end{equation*}

	The bound \eqref{eq: multiplier approximation ks metric1} holds on $\mathcal A_n$ and hence with probability at least $1-1/(2n^4)-1/n-3\sqrt{B_n^2\ln^3(pn)/n}$. In contrast, the left-hand side of \eqref{eq: gaussian approximation ks metric2} is non-stochastic. Since the preceding argument establishes \eqref{eq: gaussian approximation ks metric2} on the nonempty event $\mathcal A_n$, the bound \eqref{eq: gaussian approximation ks metric2} holds deterministically.
\rulex
	
	\vspace{2ex}

	\noindent{\bf Proof of Theorem \ref{thm: gaussian approximation main}}

	Letting $\epsilon\downarrow0$ in Corollary \ref{cor: anticoncentration} shows that $\Pr(Q_{(k)}=x)=0$ for every $x\in\mathbb R$. Therefore, $Q_{(k)}$ has a continuous distribution, and hence $\Pr(Q_{(k)}>c_{1-\alpha}^Q)=\alpha$. Applying \eqref{eq: gaussian approximation ks metric2} at the deterministic value $x=c_{1-\alpha}^Q$ yields the claimed result.
\rulex

	\noindent{\bf Proof of Theorem \ref{cor: rejection probabilities}}

	Set
	$$
	r_n=\left(\frac{B_n^2\ln^5(pn)}{n}\right)^{\frac14}\quad\text{and}\quad \eta_n=C_0k^2r_n,
	$$
	where $C_0>0$ is sufficiently large.
	As in the proof of Lemma \ref{coro: gaussian approximation}, it suffices to consider \eqref{eq: bn restriction once again}, since otherwise the conclusion follows by enlarging the implicit constant.
	By the proof of Lemma \ref{coro: gaussian approximation}, on $\mathcal A_n$,
	\begin{equation}\label{eq: gaussian quantile comparison}
		\sup_{x\in\mathbb R}\left|\Pr(G_{(k)}\leq x\mid\bs X_{1:n})-\Pr(Q_{(k)}\leq x)\right|\leq\eta_n,
	\end{equation}
	and the left-hand side of \eqref{eq: gaussian approximation ks metric2} is bounded by $\eta_n$.
	Let $u_-=(1-\alpha-\eta_n)\vee0$ and $u_+=(1-\alpha+\eta_n)\wedge1$, with the conventions $c_0^Q=-\infty$ and $c_1^Q=\infty$.
	By the definition of quantiles, \eqref{eq: gaussian quantile comparison} implies on $\mathcal A_n$ that
	\begin{equation}\label{eq: gaussian quantile sandwich}
		c_{u_-}^Q\leq c_{1-\alpha}^G\leq c_{u_+}^Q.
	\end{equation}

	Letting $\epsilon\downarrow0$ in Corollary \ref{cor: anticoncentration} shows that $\Pr(Q_{(k)}=x)=0$ for every $x\in\mathbb R$. Therefore, $Q_{(k)}$ has a continuous distribution.
	Consequently,
	$$
	\Pr(Q_{(k)}>c_{u_-}^Q)=1-u_-\leq\alpha+\eta_n
	\quad\text{and}\quad
	\Pr(Q_{(k)}>c_{u_+}^Q)=1-u_+\geq\alpha-\eta_n.
	$$
	Using the left inequality in \eqref{eq: gaussian quantile sandwich}, we obtain
	\begin{align*}
		\Pr(T_{(k)}>c_{1-\alpha}^G)
		&\leq\Pr(\{T_{(k)}>c_{1-\alpha}^G\}\cap\mathcal A_n)+\Pr(\mathcal A_n^c)\\
		&\leq\Pr(T_{(k)}>c_{u_-}^Q)+\Pr(\mathcal A_n^c)\\
		&\leq\alpha+2\eta_n+\Pr(\mathcal A_n^c).
	\end{align*}
	Similarly, the right inequality in \eqref{eq: gaussian quantile sandwich} gives
	\begin{align*}
		\Pr(T_{(k)}>c_{1-\alpha}^G)
		&\geq\Pr(\{T_{(k)}>c_{u_+}^Q\}\cap\mathcal A_n)\\
		&\geq\Pr(T_{(k)}>c_{u_+}^Q)-\Pr(\mathcal A_n^c)\\
		&\geq\alpha-2\eta_n-\Pr(\mathcal A_n^c).
	\end{align*}
	
	Finally, \eqref{eq: probability of An} and \eqref{eq: bn restriction once again} yield
	$$
	\Pr(\mathcal A_n^c)\leq\frac{1}{2n^4}+\frac{1}{n}+3\sqrt{\frac{B_n^2\ln^3(pn)}{n}}\lesssim r_n\leq\eta_n.
	$$
	Combining the preceding bounds proves the theorem.
\rulex

	\vspace{2ex}

	\subsection{Proofs of Theorem \ref{thm5.1}  and Theorem \ref{thm5.2}}
	We first give some notations. A set $ \mathcal{B}\in \mathbb{R}^p $ is a $ d $-generated polytope if it is a convex polytope with at most $ d $ facets, that is,
	$\mathcal{B} = \bigcap_{\boldsymbol v \in \mathcal{V}(\mathcal{B})} \{\boldsymbol u \in \mathbb{R}^p \mid \boldsymbol u^{\rm T} \boldsymbol v \leq S_{\mathcal{B}}(\boldsymbol v)\},	$
	in which $\mathcal{V}(\mathcal{B})$ consists of $ d $ unit vectors that are outward normal to the facets of $ \mathcal{B} $.
	For $ \epsilon > 0 $ and a $ d $-generated convex set $ \mathcal{B}^d $, define
	$\mathcal{B}^{d,\epsilon} = \bigcap_{\boldsymbol v \in \mathcal{V}(\mathcal{B}^d)} \{\boldsymbol u \in \mathbb{R}^p \mid \boldsymbol u^{\rm T} \boldsymbol v \leq S_{\mathcal{B}^d}(\boldsymbol v) + \epsilon\}.$
	It is also a $ d $-generated convex set. And recall that $\bs{X}_{1:n} = (\bs{X}_1,\ldots,\bs{X}_n)$.
	
	Without loss of generality, we may assume that $B_n^2 \leq n$ since otherwise the assertions are trivial. Let $U = \{ (x_1,\ldots,x_p) \in \mathbb{R}^p \mid \max_{1 \leq j \leq p} |x_j| \leq n\ln(2pn^{5/4})\}$. By union bound, Markov's inequality, and condition (KS1), $$
	\mathbb{P}\big( \max_{i,j} |X_{ij}| > n^{1/2}\ln(2pn^{5/4})\big)
	\leq pn \max_{i,j}\mathbb{P}\big(|X_{ij}| > n^{1/2}\ln(2pn^{5/4})\big)
	\leq pn \frac{\mathbb{E}[\exp(|X_{ij}|/B_n)]}{2pn^{5/4}} \leq \frac{1}{n^{1/4}},
	$$
	where $\max_{i,j}$ stands for $\max_{1 \leq i \leq n} \max_{1 \leq j \leq p}$ and $\sum_{i,j}$ stands for $\sum_{1 \leq i \leq n} \sum_{1 \leq j \leq p}$. Hence
	$$
	\mathbb{P}(\boldsymbol T \notin U ) \leq \mathbb{P}\big( \max_{i,j} |X_{ij}| > n^{1/2}\ln(2pn^{5/4})\big) \leq \frac{1}{n^{1/4}}.
	$$
	And under Condition (KS1) and Condition (KS2), applying Theorem 2.1 and Theorem 2.2 in \cite{ICL}, we have
	$$
	\mathbb{P}(\boldsymbol Q \notin U  )  \leq C_3 \big( \frac{B_n^2 \ln^5(pn)}{n} \big)^{\frac{1}{4}}
	$$
	and
	$$
	\mathbb{P}(\boldsymbol G \notin U \mid \bs{X}_{1:n} )  \leq C_3 \big( \frac{B_n^2 \ln^5(pn)}{n} \big)^{\frac{1}{4}},
	$$
	where $ C_3 $ is a constant depending only on $ b_1 $ and $ b_2 $.
	Hence, we get
	$$\begin{array}{rcl}
		&&\big| \mathbb{P}\big(T_\psi \leq t\big) - \mathbb{P}\big(T_\psi^{\boldsymbol Q} \leq t  \big)\big|\\
		&&\leq
		\big| \mathbb{P}\big(T_\psi \leq t ~ \text{and} ~ \boldsymbol T \in U \big) - \mathbb{P}\big(T_\psi^{\boldsymbol Q} \leq t ~ \text{and} ~ \boldsymbol Q \in U   \big) \big|
		+ C_3 \big( \frac{B_n^2 \ln^5(pn)}{n} \big)^{\frac{1}{4}}
	\end{array} $$
	and
	$$\begin{array}{rcl}
		&&\big| \mathbb{P}\big(T_\psi^{\boldsymbol Q} \leq t \big) - \mathbb{P}\big(T_\psi^{\boldsymbol G} \leq t \mid \bs{X}_{1:n}\big)\big| \\
		&&\leq
		\big| \mathbb{P}\big(T_\psi^{\boldsymbol Q} \leq t ~ \text{and} ~ \boldsymbol Q \in U  \big) - \mathbb{P}\big(T_\psi^{\boldsymbol G} \leq t ~ \text{and} ~ \boldsymbol G \in U \mid \bs{X}_{1:n} \big) \big|
		+ C_3 \big( \frac{B_n^2 \ln^5(pn)}{n} \big)^{\frac{1}{4}}.
	\end{array}$$
	Under Condition (KS4) and Condition (KS5),  we notice that $\big\{\boldsymbol x\in \mathbb{R}^p \mid \psi\big(x_{(1)},\ldots,x_{(k)}\big) \leq t\big\}=\bigcap_{\{j_1, \ldots, j_k\}\subset\{1, \ldots, p\}}\{\boldsymbol x\in \mathbb{R}^p \mid \psi\big(|x_{j_1}|, \ldots ,|x_{j_k}|\big)\leq t\}$ for $\boldsymbol x = (x_1, \ldots, x_p) \in \mathbb{R}^p$. We denote $\big\{\boldsymbol x\in \mathbb{R}^p \mid \psi\big(|x_{j_1}|, \ldots ,|x_{j_k}|\big)\leq t\big\}\cap U$ by $A_q$, where $1\leq q \leq m = \frac{p!}{(p-k)!} \leq p^k$. It suffices to bound $$\big| \mathbb{P}\big(\boldsymbol T \in \bigcap_{1\leq q \leq m}A_q\big) - \mathbb{P}\big(\boldsymbol Q \in \bigcap_{1\leq q \leq m}A_q \big) \big|$$ and
	$$\big| \mathbb{P}\big(\boldsymbol Q \in \bigcap_{1\leq q \leq m}A_q \big) - \mathbb{P}\big(\boldsymbol G \in \bigcap_{1\leq q \leq m}A_q \mid \bs{X}_{1:n}\big) \big|.$$ For brevity, we denote $n\ln(2pn^{5/4})$ by $R$. We can notice that $A_q$ are convex and compact for $q=1,\ldots,m$ by Condition (KS6).
	We can also notice that $A_q = -A_q$ for $q=1,\ldots,m$. By applying Corollary 1.2 in \cite{Barvinok2013}, we can consruct a polytope $P_q \subset \mathbb{R}^p$ with at most $(\gamma/\sqrt{\epsilon}\ln(1/\epsilon))^k$ vertices so that
	$P_q \subset A_q \subset (1+\epsilon)P_q$ for $q=1,\ldots,m$. For the sparsity of $A_q$, all vectors in $\mathcal{V}(P_q)$ can be selected to have at most k non-zero elements. Since we need at most k vertices to form a facet of $P_q$, $P_q$ has at most $(\gamma/\sqrt{\epsilon}\ln(1/\epsilon))^{k^2}$ facets for $q=1,\ldots,m$. We can deduce that $\bigcap_{1\leq q \leq m}P_q \subset \bigcap_{1\leq q \leq m}A_q \subset (1+\epsilon)\bigcap_{1\leq q \leq m}P_q$. Let $\mathcal{B}^d = \bigcap_{1\leq q \leq m}P_q$. Then $\mathcal{B}^d$ has $d \leq p^k(\gamma/\sqrt{\epsilon}\ln(1/\epsilon))^{k^2}$ facets, and all vectors in $\mathcal{V}(\mathcal{B}^d)$ have at most k non-zero elements. By Cauchy-Schwarz inequality, we have $$\mathcal{B}^d \subset \bigcap_{1\leq j \leq m}A_j \subset (1+\epsilon)\mathcal{B}^d \subset B^{d,k^{1/2}R\epsilon}.$$
	Following the proof of Proposition 3.1 in \cite{CCK17}, we let $$\begin{array}{rcl}\rho_1' = \big|\mathbb{P}(\boldsymbol T \in \mathcal{B}^d) - \mathbb{P}(\boldsymbol Q \in \mathcal{B}^d )\big| \vee \big|\mathbb{P}(\boldsymbol T \in B^{d,k^{1/2}R\epsilon}) - \mathbb{P}(\boldsymbol Q \in B^{d,k^{1/2}R\epsilon} )\big|,
	\end{array}$$
	and
	$$\begin{array}{rcl}\rho_2' = \big|\mathbb{P}(\boldsymbol G \in \mathcal{B}^d \mid \bs{X}_{1:n}) - \mathbb{P}(\boldsymbol Q \in \mathcal{B}^d )\big| \vee \big|\mathbb{P}(\boldsymbol G \in B^{d,k^{1/2}R\epsilon} \mid \bs{X}_{1:n}) - \mathbb{P}(\boldsymbol Q \in B^{d,k^{1/2}R\epsilon} )\big|.
	\end{array}$$
	We have $\mathbb{P}(\boldsymbol T \in \bigcap_{1\leq q \leq m}A_q)\leq \mathbb{P}(\boldsymbol T \in  B^{d,k^{1/2}R\epsilon})\leq \mathbb{P}(\boldsymbol Q \in  B^{d,k^{1/2}R\epsilon} )+\rho_1' $. Here notice that $(\boldsymbol v^{\rm T} \boldsymbol Q)_{\boldsymbol v\in\mathcal{V}(\mathcal{B}^d)}$ is a d-dimensional Gaussian random vector. Hence by Condition (KS3) and Lemma A.1 in \cite{CCK17}, we have
	$$
	\begin{array}{rcl}
		\mathbb{P}(\boldsymbol Q \in B^{d,k^{1/2}R\epsilon} ) & = & \mathbb{P}\{\boldsymbol v^{\rm T} \boldsymbol Q \leq S_{\mathcal{B}^d}(\boldsymbol v) + \epsilon \text{ for } \boldsymbol v \in \mathcal{V}(\mathcal{B}^d) \} \\
		& \leq & \mathbb{P}\{\boldsymbol v^{\rm T} \boldsymbol Q \leq S_{\mathcal{B}^d}(\boldsymbol v) \text{ for } \boldsymbol v \in \mathcal{V}(\mathcal{B}^d) \} + C_4k^{1/2}R\epsilon \ln^{1/2} d \\
		& = & \mathrm{P}(\boldsymbol Q \in \mathcal{B}^d ) + C_4k^{1/2}R\epsilon \ln^{1/2}d,
	\end{array}
	$$
	where $C_4$ is a constant depending on $b_1$.
	Therefore, $$\mathbb{P}(\boldsymbol T \in \bigcap_{1\leq q \leq m}A_q)\leq \mathbb{P}(\boldsymbol Q \in \bigcap_{1\leq q \leq m}A_q )+C_4k^{1/2}R\epsilon \ln^{1/2}d+\rho_1'. $$
	Likewise we have $\mathbb{P}(\boldsymbol T \in \bigcap_{1\leq q \leq m}A_q)\geq \mathbb{P}(\boldsymbol Q \in \bigcap_{1\leq q \leq m}A_q )-C_4k^{1/2}R\epsilon \ln^{1/2}d-\rho_1' $, by which we have $$\big|\mathbb{P}(\boldsymbol T \in \bigcap_{1\leq q \leq m}A_q) - \mathbb{P}(\boldsymbol Q \in \bigcap_{1\leq q \leq m}A_q) )\big|\leq C_4k^{1/2}R\epsilon \ln^{1/2}d+\rho_1'.$$
	Likewise, we get
	$$\big|\mathbb{P}(\boldsymbol G \in \bigcap_{1\leq q \leq m}A_q \mid \bs{X}_{1:n}) - \mathbb{P}(\boldsymbol Q \in \bigcap_{1\leq q \leq m}A_q )\big|\leq C_4k^{1/2}R\epsilon \ln^{1/2}d+\rho_2'.$$
	To bound $\big|\mathbb{P}(\boldsymbol T \in \mathcal{B}^d ) - \mathbb{P}(\boldsymbol Q \in \mathcal{B}^d )\big|$, we write $\mathcal{B}^d$ as $\bigcap_{\boldsymbol v \in \mathcal{V}(\mathcal{B}^d)} \{\boldsymbol x \in \mathbb{R}^p \mid \boldsymbol x^{\rm T} \boldsymbol v \leq S_{\mathcal{B}^d}(\boldsymbol v)\}$. Consider $
	\widetilde{\boldsymbol X}_{i}
	= \big( \widetilde{X}_{i1}, \ldots, \widetilde{X}_{id} \big)^{\rm T}
	= \big( \boldsymbol v^{\rm T}  \boldsymbol X_i \big)_{\boldsymbol v \in \mathcal{V}(\mathcal{B}^d)}
	$, $i = 1,\ldots, n$, where $d \leq p^k(\gamma/\sqrt{\epsilon}\ln(1/\epsilon))^{k^2}$. By replacing $\boldsymbol X_i$ by $\widetilde{\boldsymbol X}_{i}$, we can similarly define $\widetilde{\boldsymbol T}$, $\widetilde{\boldsymbol Q}$, and $\widetilde{\boldsymbol G}$.
	Then, we can find that $\mathbb{P}(\boldsymbol T \in \mathcal{B}^d)=\mathbb{P}(\widetilde{\boldsymbol T} \in B' )$, where $B' $ is in the form of $\big\{(x_1, \ldots, x_d)\in \mathbb{R}^d \mid x_s \leq a_s \text{ for } s=1,\ldots, d\big\}$ for some constants $a_s \in \mathbb{R}$. Let $N(\boldsymbol v)$ be the set consisting of positions of non-zero elements of $\boldsymbol v$. From condition (KS1), we have $\|X_{ij}\|_e \leq B_n$. Then for $\boldsymbol v \in \mathcal{V}(\mathcal{B}^d)$, $\|\boldsymbol v^{\rm T} \boldsymbol X_i\|_e \leq \sum_{j \in N(v)}\|X_{ij}\|_e \leq kB_n \leq k^2B_n$. Using Hölder' s inequality, we get $$\sum_{l=1}^{k} |x_l|
	= \sum_{l=1}^{k} |x_l| \cdot 1
	\leq \big( \sum_{l=1}^{k} |x_l|^4 \big)^{1/4}
	\big( \sum_{l=1}^{k} 1^{4/3} \big)^{3/4}
	= k^{3/4} \big( \sum_{l=1}^{k} |x_l|^4 \big)^{1/4}$$
	for $(x_1, \dots, x_k) \in \mathbb{R}^k$. From Conditions (KS2)-(KS3) and applying the inequality above, for
	$\boldsymbol v \in \mathcal{V}(\mathcal{B}^d)$, we have
	$$n^{-1} \sum_{i=1}^n\mathbb{E}[|\boldsymbol v^{\rm T} \boldsymbol X_i|^{4}] \leq n^{-1} \sum_{i=1}^n\mathbb{E}[ ( \sum_{j \in N(v)} |X_{ij}|)^{4}] \leq k^{3} n^{-1} \sum_{i=1}^n \mathbb{E}[\sum_{j \in N(v)} |X_{ij}|^{4}] \leq k^{4} B_n^{2}b_2^2.$$
	Then, applying Theorem 2.1 and Theorem 2.2 in \cite{ICL} to $\widetilde{\boldsymbol T}$, $\widetilde{\boldsymbol Q}$, and $\widetilde{\boldsymbol G}$, we get $$\big|\mathbb{P}(\boldsymbol T \in \mathcal{B}^d) - \mathbb{P}(\boldsymbol Q \in \mathcal{B}^d  )\big|\leq C_5k \big( \frac{B_n^2 \ln^5(np^k(\gamma/\sqrt{\epsilon}\ln(1/\epsilon))^{k^2})}{n} \big)^{\frac{1}{4}}, $$
	where $ C_5 $ is a constant depending only on $ b_1 $ and $ b_2 $. Likewise, we can bound $|\mathbb{P}(\boldsymbol T \in B^{d,k^{1/2}R\epsilon}) - \mathbb{P}(\boldsymbol Q \in B^{d,k^{1/2}R\epsilon} )|$, $|\mathbb{P}(\boldsymbol G \in \mathcal{B}^d \mid \bs{X}_{1:n}) - \mathbb{P}(\boldsymbol Q \in \mathcal{B}^d )|$ and $|\mathbb{P}(\boldsymbol G \in B^{d,k^{1/2}R\epsilon} \mid \bs{X}_{1:n}) - \mathbb{P}(\boldsymbol Q \in B^{d,k^{1/2}R\epsilon} )|$, so that we bound $\rho_1'$ and $\rho_2'$. As $n$ is sufficiently large, we can let $1/\epsilon=k^{1/2}n^{5/4}\ln(2pn^{5/4})$. Using triangle inequality, we give the bound of ${\rho}_1$ and ${\rho}_2$.
\rulex
	\subsection{Proof of Corollary  \ref{cor5.1}}
	(i) We consider $\|(a_1,\ldots,a_k)\|_2=1$ without loss of generality. Let $V$ be a subset of $\mathbb{R}^p$ in which for $(v_1, \ldots, v_p)\in V$, there exists $1 \leq j_1, \ldots, j_k \leq p$, such that $(v_{j_1}, \ldots, v_{j_k}) = (a_1, \ldots,a_k)$, and for $j \notin \{j_1, \ldots, j_k\}, ~ v_j = 0$. Consider $\widetilde{\boldsymbol X}_{i}
	= \big( \widetilde{X}_{i1}, \ldots, \widetilde{X}_{im} \big)^{\rm T}
	= \big( \boldsymbol v^{\rm T}  \boldsymbol X_i \big)_{\boldsymbol v \in V}
	$, for $i = 1, \ldots, n$, where $m = \frac{p!}{(p-k)!} \leq p^k$.
	Using rearrangement Inequality, we can deduce that $\psi\big(T_{(1)},\ldots,T_{(p)}\big) = \max_{1 \leq q \leq m} \frac{1}{\sqrt{n}} \sum_{i=1}^n\widetilde{X}_{iq} .$
	As is proved in the proofs of Theorem \ref{thm5.1} and Theorem \ref{thm5.2}, from Conditions (KS1)-(KS3), we get $\|\boldsymbol v^{\rm T} \boldsymbol X_i\|_e \leq k^2B_n$ and $n^{-1} \sum_{i=1}^n\mathbb{E}[|\boldsymbol v^{\rm T}  \boldsymbol X_i|^{4}] \leq k^{4} B_n^{2}b_2^2$ for $\boldsymbol v \in V$. Applying Theorem 2.1 and Theorem 2.2 in \cite{ICL}, we have
	$$\sup\limits_{t\in \mathbb{R}}\big| \mathbb{P}\big(\psi(T_{(1)},\ldots,T_{(p)}) \leq t\big) - \mathbb{P}\big(\psi(G_{(1)},\ldots,G_{(p)}) \leq t \mid \bs{X}_{1:n}\big) \big|  \lesssim k \big( \frac{B_n^2 \ln^5(p^kn)}{n} \big)^{\frac{1}{4}}, $$
	and
	$$\sup\limits_{t\in \mathbb{R}}\big| \mathbb{P}\big(\psi(T_{(1)},\ldots,T_{(p)}) \leq t\big) - \mathbb{P}\big(\psi(Q_{(1)},\ldots,Q_{(p)}) \leq t \big) \big|  \lesssim k \big( \frac{B_n^2 \ln^5(p^kn)}{n} \big)^{\frac{1}{4}}, $$
	Then, we consider $T_\psi= \psi\big(Y_{[1]},\ldots,Y_{[k]}\big)$. We only need to replace the p-dimensional vector $\boldsymbol X_i$ with the 2p-dimensional vector whose first p coordinates are equal to $\boldsymbol X_i$ and the last p coordinates are equal to $-\boldsymbol X_i$ in the scenario above. Hence, we get $$ {\rho}_1 \vee {\rho}_2 \lesssim k \big( \frac{B_n^2 \ln^5(2^kp^kn)}{n} \big)^{\frac{1}{4}}, $$ for $\psi\big(u_1, \ldots,u_k\big)=\sum_{l=1}^{k}a_lu_l$.
	(ii) We notice that $\sum_{l=1}^{k}\exp(u_l) \geq \sum_{l=1}^{k}\exp(v_l) $ when $u_l\geq v_l\geq0$ for $l = 1,\ldots,k$. We also notice that $\sum_{l=1}^{k}\exp(u_l) = \sum_{l=1}^{k}\exp(u_{(l)})$ when $u_l\geq0$ for $l = 1,\ldots,k$. Besides, we notice that $\big\{(u_1, \ldots, u_k)\in \mathbb{R}^k \mid \sum_{l=1}^{k}\exp(u_l)\leq t\big\}$ is closed and convex for $t \in \mathbb{R}$. Applying Theorem \ref{thm5.1}  and Theorem \ref{thm5.2}, we complete the proof.
	(iii) We notice that $\sum_{l=1}^{k}u_l^2 \geq \sum_{l=1}^{k}v_l^2 $ when $u_l\geq v_l\geq0$ for $l = 1,\ldots,k$. We also notice that $\sum_{l=1}^{k}u_l^2 = \sum_{l=1}^{k}u_{(l)}^2$ when $u_l\geq0$ for $l = 1,\ldots,k$. Besides, we notice that $\big\{(u_1, \ldots, u_k)\in \mathbb{R}^k \mid \sum_{l=1}^{k}u_l^2\leq t\big\}$ is closed and convex for $t \in \mathbb{R}$. Applying Theorem \ref{thm5.1}  and Theorem \ref{thm5.2}, we complete the proof.
\rulex
	\subsection{Proof of Theorem \ref{thm6.1}}
	We consider $$T_F^{\bs G}=\min\big\{1- F_{l_1}\big( T_{\psi_{l_1}}^{\bs G}\big), \ldots ,1- F_{l_s}\big(T_{\psi_{l_s}}^{\bs G}\big)\big\}, $$
	where  $ F_l(t)$  is the cumulative distribution function  of $T_{\psi_l}^{\bs G}\mid \bs{X}_{1:n}$
	for $l=l_1, \ldots, l_s$ with $1\leq l_1 < \ldots < l_s=k$. We first prove that $$ \sup\limits_{t\in \mathbb{R}}\big| \mathbb{P}\big(T_F < t\big) - \mathbb{P}\big(T_F^{\boldsymbol{G}}  < t\mid \bs{X}_{1:n} \big) \big|\lesssim k^{\frac{7}{2}} \Big( \frac{B_n^2 \ln^5(\gamma k^{\frac{1}{4}}pn^{\frac{13}{8}}\ln^{\frac{1}{2}}(2pn^{\frac{5}{4}})\ln[k^{\frac{1}{2}}n^{\frac{5}{4}}\ln(2pn^{\frac{5}{4}})])}{n} \Big)^{\frac{1}{4}}$$ when Conditions (KS1)-(KS3) are satisfied and $n$ is sufficiently large. Then we prove that $$\lim_{B\rightarrow \infty }\big|\frac{1}{B}\sum_{r=1}^B\mathds{1}{\{t_r<t\}}-\mathbb{P}(T_F^{\bs G}< t\mid \bs{X}_{1:n} )\big|=0 $$ almost surely for $t\in \mathbb{R}$. Combining these two results and applying triangle inequality completes the proof of Theorem \ref{thm6.1}.
	We still assume that $B_n^2 \leq n$ and use the same notations $U$ and $R$ as that in the proofs of Theorem \ref{thm5.1} and Theorem \ref{thm5.2}. In the first step, $$\sup\limits_{t\in \mathbb{R}}\big| \mathbb{P}\big(T_F < t\big) - \mathbb{P}\big(T_F^{\boldsymbol{G}}  < t \mid \bs{X}_{1:n} \big) \big|=\sup\limits_{t\in \mathbb{R}}\big| \mathbb{P}\big(T_F \geq t\big) - \mathbb{P}\big(T_F^{\boldsymbol{G}} \geq t \mid \bs{X}_{1:n} \big) \big|
	.$$ For $l=l_1, \ldots, l_s$, we notice that $\big\{(u_1,\ldots,u_k)^{\rm T} \in \mathbb{R}^p \mid 1-F_l\big(\psi_l(u_{[1]},\ldots,u_{[l]})\big) \geq t\big\}$ can be rewritten as  $\bigcap_{\{j_1, \ldots, j_l\}\subset\{1, \ldots, p\}}\big\{(u_1,\ldots,u_k)^{\rm T} \in \mathbb{R}^p \mid \psi_l(|u_{j_1}|, \ldots ,|u_{j_l}|)\leq F_l^{-1}(1-t)\big\}$ where $F_l^{-1}(1-t)=\inf_{x\in \mathbb{R}}\{F_l(x)\geq 1-t \}$. And we denote $\big\{\boldsymbol u\in \mathbb{R}^p \mid \psi_l(|x_{j_1}|, \ldots ,|x_{j_l}|)\leq F_l^{-1}(1-t)\big\}\cap U$ by $A_{lq}$, where $1\leq q \leq m_l = \frac{p!}{(p-l)!} \leq p^l$. Following the proofs of Theorem \ref{thm5.1} and Theorem \ref{thm5.2}, it suffices to bound $$\bigg| \mathbb{P}\big(\boldsymbol T \in\bigcap_{l=l_1,\ldots,l_s}\bigcap_{1\leq q \leq m_l}A_{lq}\big) - \mathbb{P}\big(\boldsymbol G \in \bigcap_{l=l_1, \ldots, l_s}\bigcap_{1\leq q \leq m_l}A_{lq}\mid \bs{X}_{1:n}\big) \bigg|.$$
	We can notice that $A_{lq}$ are convex and compact and $A_{lq}=-A_{lq}$.
	For $l=l_1, \ldots, l_s$ and $1\leq q\leq m_l$,  there exists $1\leq \tilde{q} \leq m_k$ such that the main dimension of $A_{k\tilde{q}}$ contains that of $A_{lq}$. For example, the main dimension of $\{\boldsymbol u \in \mathbb{R}^p \mid \psi_l(|u_{j_1}|, \ldots ,|u_{j_l}|)\leq F_i^{-1}(t)\}\cap U$ is $\{j_1, \ldots, j_l\}$. Therefore, we can find all the $A_{lq}$ with $l=l_1, \ldots, l_s$ and $1\leq q\leq m_l$ whose main dimension is contained in that of $A_{k\tilde{q}}$ with $1\leq \tilde{q} \leq m_k$. Let their intersection be $\tilde{A}_{k\tilde{q}}$ which is also convex and compact and $\tilde{A}_{k\tilde{q}}=-\tilde{A}_{k\tilde{q}}$. Then, we have $$\bigcap_{l=l_1, \ldots, l_s}\bigcap_{1\leq q \leq m_l}A_{lq}=\bigcap_{1\leq q \leq m_k}\tilde{A}_{kq}. $$
	Applying Theorem \ref{thm5.2}, we have $$\bigg| \mathbb{P}\big(\boldsymbol T \in\bigcap_{1\leq q \leq m_k}\tilde{A}_{kq}\big) - \mathbb{P}\big(\boldsymbol G \in \bigcap_{1\leq q \leq m_k}\tilde{A}_{kq}\mid \bs{X}_{1:n}\big) \bigg|\lesssim k^{\frac{7}{2}} \big( \frac{B_n^2 \ln^5(\gamma k^{\frac{1}{4}}pn^{\frac{13}{8}}\ln^{\frac{1}{2}}(2pn^{\frac{5}{4}})\ln[k^{\frac{1}{2}}n^{\frac{5}{4}}\ln(2pn^{\frac{5}{4}})])}{n} \big)^{1/4}, $$
	which completes the proof of the first step.
	In the next step, applying Glivenco-Cantelli Theorem, we have $$\lim_{B\rightarrow \infty}\frac{1}{B-1}
	\sum_{\substack{h=1 \\ h\neq r}}^B\mathds{1}\big\{T_{\psi_{l}}^{\bs G^{(h)}} > T_{\psi_{l}}^{\bs G^{(r)}}\big\}=F_l\big(T_{\psi_{l}}^{\bs G^{(r)}}\big)$$ almost surely
	for $r= 1,\ldots,B$ and $l= l_1,\ldots, l_s$. Therefore, we get $$\lim_{B\rightarrow \infty}t_r=\min\big\{1-F_{l_1}\big(T_{\psi_{l_1}}^{\bs G^{(r)}}\big), \ldots, 1-F_{l_s}\big(T_{\psi_{l_s}}^{\bs G^{(r)}}\big)\big\}$$ alomst surely for $r= 1,\ldots,B$.
	We observe that $$\bigg|\frac{1}{B-1}
	\sum_{\substack{h=1 \\ h\neq r_1}}^B\mathds{1}\big\{T_{\psi_{l}}^{\bs G^{(h)}} > t\big\}-\frac{1}{B-1}
	\sum_{\substack{h=1 \\ h\neq r_2}}^B\mathds{1}\big\{T_{\psi_{l}}^{\bs G^{(h)}} > t\big\} \bigg| \leq \frac{2}{B-1}$$ for $r_1, r_2= 1,\ldots,B$ , $l= l_1,\ldots, l_s$, and $t\in \mathbb{R}$.
	Therefore, for $\epsilon > 0$, there exists $B$ such that $$\big|t_r-\min\big\{1-F_{l_1}\big(T_{\psi_{l_1}}^{\bs G^{(r)}}\big), \ldots, 1-F_{l_s}\big(T_{\psi_{l_s}}^{\bs G^{(r)}}\big)\big\} \big| \leq \epsilon $$
	almost surely for $r= 1,\ldots,B$. And we can notice that $$\frac{1}{B}\sum_{r=1}^B\mathds{1}{\{t_r<t\}}=\frac{1}{B}\sum_{r=1}^B\mathds{1}\big\{\min\big\{1-F_{l_1}\big(T_{\psi_{l_1}}^{\bs G^{(r)}}\big), \ldots, 1-F_{l_s}\big(T_{\psi_{l_s}}^{\bs G^{(r)}}\big)\big\}<t\big\}$$
	almost surely for $t\in \mathbb{R}$ and sufficiently small $\epsilon$. Besides, applying strong law of large numbers, we have $$\lim_{B\rightarrow \infty}\frac{1}{B}\sum_{r=1}^B\mathds{1}\big\{\min\big\{1-F_{l_1}\big(T_{\psi_{l_1}}^{\bs G^{(r)}}\big), \ldots, 1-F_{l_s}\big(T_{\psi_{l_s}}^{\bs G^{(r)}}\big)\big\}<t\big\}=\mathbb{P}\big(T_F^{\bs G} < t \mid \bs{X}_{1:n} \big)$$ alomst surely for $t\in \mathbb{R}$. Thus, we completes the proof of the second step.
\rulex
	\vspace{2ex}
	\section*{Acknowledgments}
    We thank Jake Carlson of the Department of Economics at Harvard University for his helpful comments.
	The authors are in an alphabetic order. Qizhai Li is partially supported by
	National Nature Science Foundation of China (12325110,12288201) and the robotic AI-Scientist platform of Chinese Academy of Sciences.
	Liuquan Sun is supported by the National Natural Science Foundation of China (No.12171463).
	\renewcommand\bibname{References}
	
\end{document}